\documentclass[11pt]{amsart}
\usepackage[utf8]{inputenc}
\usepackage{amssymb,amsmath}
\usepackage{graphicx}
\usepackage{color}
\usepackage{tikz}
\usepackage{pgfplots}
\usepackage{enumitem}
\usepackage{algorithm}
\usepackage{setspace}
\usepackage{algpseudocode}
\usepackage{mathtools}
\usepackage{xfrac}

\usepackage{accents}
\usepackage{multicol} 
\usepackage{soul}
\usepackage{subcaption}
\usepackage{booktabs}
\usepackage{array}
\newcolumntype{L}[1]{>{\raggedright\let\newline\\\arraybackslash\hspace{0pt}}m{#1}}
\newcolumntype{C}[1]{>{\centering\let\newline\\\arraybackslash\hspace{0pt}}m{#1}}
\newcolumntype{R}[1]{>{\raggedleft\let\newline\\\arraybackslash\hspace{0pt}}m{#1}}
\usepackage{siunitx}
\usepackage{bbold}

\usepackage[hidelinks]{hyperref}
\usepackage[nameinlink,noabbrev]{cleveref}
\newtheorem{theorem}{Theorem}
\newtheorem{proposition}[theorem]{Proposition}

\theoremstyle{definition}

\newtheorem{example}[theorem]{Example}
\theoremstyle{lemma}
\newtheorem{lemma}[theorem]{Lemma}
\newtheorem{coroll}[theorem]{Corollary}
\theoremstyle{remark}
\newtheorem{remark}[theorem]{Remark}

\newtheorem{assumption}[theorem]{Assumption}
\Crefname{assumption}{Assumption}{Assumptions}
\numberwithin{theorem}{section}
\numberwithin{equation}{section}
\numberwithin{table}{section}
\numberwithin{figure}{section}
\definecolor{myBlue}{RGB}{30,144,255} 
\definecolor{myGreen}{RGB}{69,169,0} 
\definecolor{myRed}{RGB}{165,12,42} 
\definecolor{myOrange}{RGB}{225,92,22} 
\definecolor{color0}{rgb}{0.12156862745098,0.466666666666667,0.705882352941177}
\definecolor{color1}{rgb}{1,0.498039215686275,0.0549019607843137}
\definecolor{color2}{rgb}{0.172549019607843,0.627450980392157,0.172549019607843}
\definecolor{color3}{rgb}{0.83921568627451,0.152941176470588,0.156862745098039}
\definecolor{color4}{rgb}{0.580392156862745,0.403921568627451,0.741176470588235}
\definecolor{color5}{rgb}{0,0,0}

\newcommand{\delete}[1]{ }

\def\N{\mathbb{N}}

\def\R{\mathbb{R}}

\newcommand\ds{\,\text{d}s}

\newcommand\dx{\,\text{d}x}

\newcommand{\tddt}{\ensuremath{\tfrac{\text{d}}{\text{d}t}} }

\newcommand{\hook}{\ensuremath{\hookrightarrow}}

\DeclareMathOperator{\ddiv}{div}

\newcommand{\calA}{\ensuremath{\mathcal{A}} }

\newcommand{\calB}{\ensuremath{\mathcal{B}} }
\newcommand{\calC}{\ensuremath{\mathcal{C}} }
\newcommand{\calD}{\ensuremath{\mathcal{D}} }

\newcommand{\cHV}{\ensuremath{{\mathcal{H}_{\scalebox{.5}{\V}}}}}
\newcommand{\cHQ}{\ensuremath{{\mathcal{H}_{\scalebox{.5}{\Q}}}}}

\newcommand{\cHZ}{\ensuremath{{\mathcal{H}_{\scalebox{.5}{\cZ}}}}}

\newcommand{\calK}{\ensuremath{\mathcal{K}} }

\newcommand{\Q}{\ensuremath{\mathcal{Q}} }

\newcommand{\V}{\ensuremath{\mathcal{V}}}

\newcommand{\calX}{\ensuremath{\mathcal{X}} }

\newcommand{\cZ}{\ensuremath{\mathcal{Z}} }

\newcommand{\eps}{\ensuremath{\varepsilon}}

\newcommand{\dd}{\ensuremath{\hat d}}
\newcommand{\m}{\ensuremath{m}}

\newcommand{\proj}{R}
\newcommand{\projInd}[1]{\ensuremath{\proj}_{#1}}
\newcommand{\Ri}{\projInd{i}}
\newcommand{\Rj}{\projInd{j}}
\newcommand{\Ru}{\projInd{u}}
\newcommand{\Rp}{\projInd{p}}

\newcommand{\CQtoH}{\ensuremath{{C_{\scalebox{.5}{\Q \hook \cHQ}}}} }
\newcommand{\CQtoHsquare}{\ensuremath{{C^2_{\scalebox{.5}{\Q \hook \cHQ}}}} }

\newcommand{\mos}[2]{{#1}_{I_{#2}}}
\newenvironment{smallbmatrix}{\left[\begin{smallmatrix}}{\end{smallmatrix}\right]}
\newcommand{\gt}{\widetilde{g}}

\newcommand{\mosp}[1]{\mos{\bar{p}}{#1}}
\newcommand{\mosdp}[1]{\mos{\dot{\bar{p}}}{#1}}
\newcommand{\LinfHQ}{L^\infty(0,\tau;\cHQ\!)}
\newcommand{\LinfQd}{L^\infty(0,\tau;\Q^*\!)}

\newcommand{\Kalpha}{\kappa_{\Q}}
\newcommand{\Kbeta}{\kappa_{\cHQ\!}}
\allowdisplaybreaks
\begin{document}
\title[Semi-explicit Discretization Schemes for Elliptic-Parabolic Problems]{Semi-explicit Discretization Schemes for Weakly-Coupled Elliptic-Parabolic Problems}
\author[]{R.~Altmann$^\dagger$, R.~Maier$^\dagger$, B.~Unger$^{\ddagger}$}
\address{${}^{\dagger}$ Department of Mathematics, University of Augsburg, Universit\"atsstr.~14, 86159 Augsburg, Germany}
\email{\{robert.altmann,roland.maier\}@math.uni-augsburg.de}
\address{${}^{\ddagger}$ Institute of Mathematics MA\,{}4-5, Technical University Berlin, Stra\ss e des 17.~Juni 136, 10623 Berlin, Germany}
\email{unger@math.tu-berlin.de}
%
\date{\today}
%
%
\begin{abstract}
We prove first-order convergence of the semi-explicit Euler scheme combined with a finite element discretization in space for elliptic-parabolic problems which are weakly coupled. This setting includes poroelasticity, thermoelasticity, as well as multiple-network models used in medical applications. The semi-explicit approach decouples the system such that each time step requires the solution of two small and well-structured linear systems rather than the solution of one large system. The decoupling improves the computational efficiency without decreasing the convergence rates. The presented convergence proof is based on an interpretation of the scheme as an implicit method applied to a constrained partial differential equation with delay term. Here, the delay time equals the used step size. This connection also allows a deeper understanding of the weak coupling condition, which we accomplish to quantify explicitly. 
\end{abstract}
%
%
\maketitle
%
%
%
%
\section{Introduction}
We study the semi-explicit time discretization of a linear elliptic problem that is coupled to a linear parabolic equation, which we refer to as \emph{elliptic-parabolic problem}. The resulting model is a partial differential-algebraic equation (PDAE) that appears, for instance, in the field of geomechanics~\cite{Biot41, Zob10}. In particular, we consider the deformation of porous media saturated by an incompressible viscous fluid, also called \emph{poroelasticity} \cite{DetC93, Sho00}. The displacement of a material due to temperature changes gives a second application, which is commonly known as \emph{thermoelasticity} \cite{Bio56}. These applications have in common that the scaling of the coupling term is typically small, which we refer to hereafter as \emph{weakly coupled}. 

An alternative formulation of poroelasticity is obtained by introducing the fluid flux, also called \emph{Darcy velocity}, as an additional variable. This so-called \emph{three-field formulation} is used, for instance, in biomechanics to predict the deformation resulting from tumor growth in the brain \cite{RNM+03}. This model is advantageous if one is particularly interested in the fluid flux, since no subsequent calculation is needed. 
Further, it is well-suited for the extension to network structures, which are used, for instance, in medical applications with several pressure variables. As an example, we mention the investigation of cerebral edema, which may occur as a result of an unnatural accumulation of cerebrospinal fluid in the brain (hydrocephalus) \cite{VarCTHLTV16}. There, the brain is modeled as a poroelastic medium saturated by four fluid networks: one for the high-pressure arteries, one for the low-pressure arterioles and capillaries, one for the cerebrospinal and interstitial fluids, and one for the veins. 

As mentioned above, we deal with PDAEs such that a semi-discretization in space yields a differential-algebraic equation (DAE). 
As an immediate consequence, one cannot use explicit time-integration schemes \cite{KunM06}.  
The current literature mainly considers a time discretization by the implicit Euler scheme. 
In~\cite{ErnM09,MalP17} this is combined with a finite element discretization in space. The performed error analysis is based on a spatial projection, which is related to the corresponding stationary problem and thus coupled. 
A decoupled projection operator is introduced in~\cite{AltCMPP19, FuACMPP19} for heterogeneous poroelasticity. 
Other spatial discretization schemes such as continuous and discontinuous Galerkin methods are considered in~\cite{PhiW07a, PhiW07b, PhiW08}. 
Numerical methods based on the three-field formulation are discussed, e.g., in~\cite{HuRGZ17, HonK18, HonKLP19}. 
Finally, we mention~\cite{Fu19} where higher-order schemes in space and time are investigated. 

We emphasize that all mentioned schemes rely on an implicit time discretization. Thus, one needs to solve a large coupled system in each time step. This may be resolved by using a semi-explicit time-stepping method. 
Such a discretization decouples the elliptic and parabolic equation with the obvious advantage that the system is split into two subsystems with smaller dimensions. Further, the sparsity pattern of the matrices improves such that the construction of block-preconditioners is facilitated~\cite{LMW17}.  
A first attempt in this direction is presented in~\cite{WheG07} for the three-field model. There, however, an additional inner iteration is necessary to guarantee convergence. 
For yet another three-field model, \cite{JCLT19} proposes a similar semi-explicit scheme as in the present paper. 
However, there is no convergence analysis available. 

This paper provides theoretical justification for the decoupling of the elliptic and parabolic equation. We prove convergence of the semi-explicit Euler discretization in time combined with any stable spatial discretization of first order. 
This includes the classical two-field formulation (\Cref{thm:twoFieldConvergence}) as well as the multiple-network case if the exchange rates are small enough (\Cref{thm:networkConvergence}). 
Besides suitable regularity assumptions, we require a weak coupling condition, which is motivated and introduced in \Cref{sec:model}. The convergence proof is based on decoupled spatial projections and the observation that the semi-explicit scheme equals the implicit discretization of a related PDAE with delay term. Hereby, the fixed time delay~$\tau$ equals the step size for the time integration. 

The key technique for the convergence result is to prove that the original coupled PDAE and the associated delay PDAE only differ by an order~$\tau$, see \Cref{prop:delayTwoField} and \Cref{prop:delayNetwork} for further details. This novel proof technique allows us to establish the expected rates in space and time and an explicit quantification of the weak coupling condition. As an additional benefit from the delay approach, we observe that the weak coupling condition resembles a necessary condition for asymptotic stability of the semi-discretized delay DAE. 
This fact is also illustrated in the numerical experiments of~\Cref{sec:num}, showing that the stated condition is indeed sharp. 
We foresee that this strategy can be extended to study further time discretization schemes of semi-explicit type. 
\subsection*{Notation} 
Throughout the paper we write~$a \lesssim b$ to indicate that there exists a generic constant~$C$, independent of spatial and temporal discretization parameters, such that~$a \leq C b$. 
Further, we abbreviate Bochner spaces on the time interval~$[0,T]$ for a Banach space~$\calX$ by~$L^p(\calX) \vcentcolon = L^p(0,T;\calX)$, $W^{k,p}(\calX) \vcentcolon = W^{k,p}(0,T;\calX)$, and~$H^k(\calX) \vcentcolon = H^k(0,T;\calX)$, $p\ge 1$, $k\in\N$. 
%
%
\section{Elliptic-parabolic Problems}
\label{sec:model}
This section is devoted to an introduction to the considered elliptic-parabolic problems. 
For this, we consider the weak formulation of the classical two-field model as well as three-field and multiple-network systems. 
To keep the models fairly general we consider abstract formulations and gather all needed assumptions on the functions spaces and the involved bilinear forms. 
However, we also discuss practical examples for each case. 
These examples motivate the notion of \emph{weak coupling}, which we specify in terms of the system parameters.
%
%
\subsection{Two-field formulation}\label{sec:model:two}
We consider the weak formulation of elliptic-parabolic problems with two unknowns~$u\colon[0,T] \to \V$ and~$p\colon[0,T] \to \Q$, where~$T<\infty$ denotes the final time and $\V$, $\Q$ are Hilbert spaces which already include the boundary conditions, see the examples below.  
In the abstract setting, the solution pair $(u,p)$ should satisfy 
\begin{subequations}
\label{eqn:two}
\begin{align}
    a(u,v) - d(v, p) 
    &= \langle f, v \rangle, \label{eqn:two:a} \\
    d(\dot u, q) + c(\dot p,q) + b(p,q) 
    &= \langle g, q\rangle \label{eqn:two:b} 
\end{align}
for all test functions~$v\in \V$, $q \in \Q$ and sufficiently smooth source terms~$f\colon[0,T] \to \V^*$, $g\colon[0,T] \to \Q^*$. 
Hereby, $\V^*$ and $\Q^*$ denote the respective dual spaces for $\V$ and $\Q$ and $\langle\, \cdot\,,\,\cdot\,\rangle$ denotes the duality pairing. 
Further, we have initial conditions  
\begin{align}
	u(0) = u^0 \in \V, \qquad
    p(0) = p^0 \in \Q, \label{eqn:two:c}
\end{align}
\end{subequations}
which need to respect a consistency condition since system~\eqref{eqn:two} defines a PDAE. Although system~\eqref{eqn:two} is not in the standard semi-explicit form as analyzed in~\cite{EmmM13,Alt15,AltH18}, the consistency condition is explicitly given by equation~\eqref{eqn:two:a} and reads
\begin{displaymath}
	a(u^0,v) - d(v, p^0) = \langle f(0),v\rangle
\end{displaymath}
for all~$v\in \V$. 
We further assume that both ansatz spaces are part of a Gelfand triple, cf.~\cite[Ch.~23.4]{Zei90a}. For this, we introduce the pivot spaces~$\cHV$ and~$\cHQ$ such that~$\V, \cHV, \V^*$ and~$\Q, \cHQ, \Q^*$ each form a Gelfand triple. 
A typical example considers Sobolev spaces including the first weak derivative for~$\V$, $\Q$ and standard~$L^2$-spaces for~$\cHV$, $\cHQ$. 

For the involved bilinear forms~$a$, $b$, $c$, and $d$ we make the following assumptions: 
the bilinear form~$a\colon \V\times\V\to \R$ is symmetric, elliptic, and bounded in~$\V$, i.e.,  
\begin{equation*}
  a(u,u) 
  \geq c_a\Vert u \Vert_\V^2, \qquad
  a(u,v) 
  \leq C_a\Vert u \Vert_\V \Vert v\Vert_\V
\end{equation*}
for all~$u, v\in \V$. Note that this also defines a norm~$\Vert\cdot\Vert_a \vcentcolon= a(\cdot, \cdot)^{1/2}$, which is equivalent to the~$\V$-norm. 
Similarly, $b\colon \Q\times\Q\to \R$ is symmetric, elliptic, and bounded in~$\Q$, i.e.,  
\begin{equation*}
  b(p,p) 
  \geq c_b\Vert p \Vert_\Q^2, \qquad
  b(p,q)  
  \leq C_b\Vert p \Vert_\Q \Vert q\Vert_\Q
\end{equation*}
for all~$p, q\in \Q$, defining the norm~$\Vert\cdot\Vert_b$, which is equivalent to the~$\Q$-norm. 
In the examples in mind where $\V$ and~$\Q$ are Sobolev spaces, we assume additionally that the elliptic problems corresponding to~$a$ and~$b$ are~$H^2$-regular, cf.~\cite[Sect.~II.7]{Bra07}. Note that this includes an implicit condition on the spatial domain on which the bilinear forms are defined. 

Further, the bilinear form~$c\colon \Q\times\Q\to \R$ is symmetric, elliptic, and bounded in the pivot space~$\cHQ$, i.e.,  
\begin{equation*}
  c(p,p) 
  \geq c_c \Vert p \Vert_{\cHQ}^2, \qquad
  c(p,q) 
  \leq C_c \Vert p \Vert_\cHQ \Vert q\Vert_\cHQ
\end{equation*}
for all~$p, q\in \cHQ$. 
Thus, this defines a norm~$\Vert\cdot\Vert_c$, which is equivalent to the~$\cHQ$-norm
Finally, the coupling is defined through the bilinear form~$d\colon \V\times\Q\to \R$, which is bounded in terms of
\begin{equation*}
  d(u,p) 
  \leq C_d \Vert u \Vert_\V \Vert p\Vert_\cHQ, \qquad
  d(u,p) 
  \leq \tilde C_d \Vert u \Vert_\cHV \Vert p\Vert_\Q
\end{equation*}
for all~$u \in \V$ and $p\in \Q$. The possibility to choose whether to estimate $u$ or $p$ in the stronger norm will be used in the convergence analysis in \Cref{subsec:twoFieldDelayConvergence}.

We emphasize that the assumptions on the bilinear forms~$a$ and~$b$ imply the elliptic nature of equation~\eqref{eqn:two:a} and the parabolic nature of equation~\eqref{eqn:two:b}, which are coupled through the bilinear form~$d$. Furthermore, it is sufficient to prescribe an initial condition for~$p$ as equation~\eqref{eqn:two:a} is then uniquely solvable for~$u^0$.
\begin{remark}[DAE structure]
\label{rem:two:DAE}
Since~\eqref{eqn:two} represents a PDAE, a spatial discretization with parameter~$h$ leads to a DAE. Considering the time derivative of the first equation, the semi-discrete system can be written as 
\[
\begin{bmatrix} K_a & -D^T \\ D & M_c \end{bmatrix}
\begin{bmatrix} \dot u_h \\ \dot p_h \end{bmatrix}
=  
\begin{bmatrix} 0 & 0 \\ 0 & -K_b \end{bmatrix}
\begin{bmatrix} u_h \\ p_h \end{bmatrix}
+ \begin{bmatrix} \dot f_h \\ g_h \end{bmatrix}.
\]
Here, $K_a$ and~$K_b$ denote the stiffness matrices corresponding to the bilinear forms $a$ and~$b$, respectively, and~$M_c$ is the mass matrix resulting from~$c$. Under reasonable assumptions on the spatial discretization, the properties stated above imply that these three matrices are positive definite. 
Since the block-diagonal part of the matrix on the left is positive definite and the off-diagonal part is skew-symmetric, the matrix on the left is invertible, which implies that the original DAE has index~$1$, cf.~\cite[Sect.~2.2]{BreCP96}. In other words, only a single time derivative is necessary in order to reformulate the semi-discrete system as an ODE. 
\end{remark}
\begin{example}[Poroelasticity]
\label{exp:poro}
A well-known example, which fits in the framework of this subsection, is the linear poroelasticity problem in a bounded Lipschitz domain $\Omega \subseteq \mathbb{R}^d$ with $d\in\{2,3\}$, cf.~\cite{Biot41,Sho00}. In this application, we seek for the displacement field~$u\colon[0,T] \times \Omega \rightarrow \mathbb{R}$ and the pressure~$p\colon[0,T] \times \Omega \rightarrow \mathbb{R}$. Considering homogeneous Dirichlet boundary conditions, we have 
\begin{align*}
  a(u,v) \vcentcolon= \int_\Omega \sigma(u) : \varepsilon(v) \dx, \qquad 
  b(p,q) \vcentcolon= \int_\Omega \frac{\kappa}{\nu}\, \nabla p \cdot \nabla q \dx,\\
  c(p,q) \vcentcolon= \int_\Omega \frac{1}{M}\, p\, q \dx, \qquad 
  d(u,q) \vcentcolon= \int_\Omega \alpha\, (\nabla \cdot u)\, q \dx
\end{align*}
with spaces
\[
  \V:=[H^1_0(\Omega)]^d, \qquad
  \cHV:=[L^2(\Omega)]^d, \qquad 
  \Q:=H^1_0(\Omega), \qquad 
  \cHQ:=L^2(\Omega).
\]
The involved parameters include the stress tensor~$\sigma$ (defined by the Lam\'e coefficients $\lambda$ and $\mu$), the permeability~$\kappa$, the Biot-Willis fluid-solid coupling coefficient~$\alpha$, the Biot modulus~$M$, and the fluid viscosity~$\nu$. 
As usual in linear elasticity, $\varepsilon(u)$ denotes the symmetric gradient. 
The source terms satisfy~$f \equiv 0$ and~$g$ represents an injection or production process. 
We emphasize that the ellipticity of the bilinear form~$a$ follows from Korn's inequality. The resulting ellipticity constant is given by~$\mu$, see e.g.~\cite[Th.~6.3.4]{Cia88} for details.  
In many applications the coupling coefficient~$\alpha$ is smaller than one and thus much smaller than the Lam\'e coefficients, see e.g.~\cite[Sect.~3.3.4]{DetC93}. 
\end{example}
\begin{example}[Thermoelasticity]
\label{exp:thermo}
Since the linear thermoelastic problem is -- in mathematical terms -- equivalent to linear poroelasticity, system~\eqref{eqn:two} also applies to this case, cf.~\cite{Bio56}. 
Thermoelasticity describes the displacement of a material due to temperature changes. Similar to~\Cref{exp:poro}, the thermal expansion coefficient in the bilinear form~$d$, which is responsible for the coupling, is much smaller than the stress tensor, cf.~\cite{CalR14}. 
\end{example}
Motivated from the previous examples, we make the following assumption on the coupling of the elliptic and parabolic equation. 
\begin{assumption}[Weak coupling]
\label{ass:weakCouplingTwo}
We assume a weak coupling through the bilinear form~$d$ in the sense that
\begin{align*}
  C_d^2 \le c_a\, c_c. 
\end{align*}
\end{assumption}
As already mentioned and indicated in the two examples presented above, this assumption is satisfied in many applications. A detailed list of poroelastic parameters for different stones is given in~\cite{DetC93}. 
%
%
\subsection{Three-field formulation}\label{sec:model:three}
System~\eqref{eqn:two} can also be expressed in a three-field formulation, i.e., with an additional variable reflecting the flux of~$p$. Such a formulation may be beneficial if the flux is of particular interest and serves here as a first step in the direction of network models. 
For this formulation, we need three spaces, namely~$\V$, $\cZ$, and $\Q$, and aim to find the three unknowns~$u\colon[0,T] \to \V$, $y\colon[0,T] \to \cZ$, and~$p\colon[0,T] \to \Q$, which satisfy the system
\begin{subequations}
\label{eqn:three}
\begin{align}
	a(u,v) - d(v, p) 
	&= \langle f, v \rangle, \label{eqn:three:a} \\
	(y,z)_{\cHZ} - \dd(z,p) 
	&= 0, \label{eqn:three:b} \\
	d(\dot u, q) + c(\dot p,q) + \dd(y,q) 
	&= \langle g, q\rangle \label{eqn:three:c} 
\end{align}
\end{subequations}
for all test functions~$v\in \V$, $z\in \cZ$, and $q \in \Q$. 
The corresponding initial condition reads~$p(\,\cdot\,,0) = p^0 \in \Q$, which defines~$u^0\in\V$ through equation~\eqref{eqn:three:a}. 

In the three-field formulation, we assume Gelfand triples~$\V, \cHV, \V^*$ and $\cZ, \cHZ, \cZ^*$. The space~$\Q$ is typically an $L^2$-space such that no pivot space is needed or, in other words, $\Q=\cHQ$. 
For the bilinear forms~$a$ and~$c$, we have the same assumptions as in the previous subsection. Note, however, that the assumptions on~$\Q$ imply that~$c$ is now elliptic on~$\Q$. Further, $d\colon \V\times\Q\to \R$ is bounded such that 
\begin{equation*}
d(u,p) 
\leq C_d \Vert u \Vert_\V \Vert p\Vert_\Q 
\end{equation*}
for all~$u \in \V$ and $p\in \Q$. For the newly introduced bilinear form~$\dd\colon \cZ\times \Q \to \R$ we assume continuity in the sense of
\[
\dd(y,p) 
\le C_{\dd} \Vert y \Vert_\cZ \Vert p\Vert_\Q
\]
for all~$y \in \cZ$ and $p\in \Q$. 
\begin{example}[Poroelasticity]
\label{exp:poro:threeField}
We revisit Example~\ref{exp:poro}. 
The corresponding three-field formulation seeks for the displacement~$u$, the fluid flux or Darcy velocity~$y$, and the pore pressure~$p$, cf.~\cite{Biot41}. In the strong form the system reads
\begin{subequations}
\label{eqn:poro:threeField}
\begin{align}
	-\nabla \cdot ( \sigma (u) ) + \nabla (\alpha p) 
	&= f, \\
	y + \tfrac{\kappa}{\nu}\, \nabla p 
	&= 0, \\
	- \alpha\, \nabla \cdot \dot u - \tfrac{1}{M} \dot p - \nabla \cdot y 
	&= g
\end{align}    
\end{subequations}
with an initial condition for~$p$. 
The natural boundary conditions in this setting are homogeneous Dirichlet boundary conditions for the displacement and homogeneous Neumann boundary conditions for the pressure that can be reformulated as a boundary condition for the fluid flux. 
In this case, the spaces are given by 
\[ 
  \V :=  H^1_0(\Omega)]^d, \quad 
  \cHV := [L^2(\Omega)]^d, \quad
  \cZ := H_0(\ddiv, \Omega), \quad
  \cHZ := [L^2(\Omega)]^d, \quad
  \Q := L^2(\Omega).
\]
Here, $H_0(\ddiv, \Omega)$ denotes the space of functions~$z \in [L^2(\Omega)]^d$ with $\nabla\cdot z \in L^2(\Omega)$ and~$z\cdot n = 0$ on~$\partial \Omega$ with unit normal vector~$n$. 
In order to guarantee uniqueness of the pressure, one often assumes some additional condition such as a vanishing integral.  	 
This example fits in the framework of~\eqref{eqn:three} with 
\[
  \dd(z,p) := \int_\Omega \Big(\nabla \cdot \big(\sqrt{\tfrac{\kappa}{\nu}}\, z\big)\Big)\, p \dx.
\]
and~$(\,\cdot\,,\,\cdot\,)_\cHZ$ as the standard~$L^2$-norm. 
\end{example}
At this point, it would be reasonable to consider a similar weak coupling condition as in~\Cref{ass:weakCouplingTwo}. However, we will discuss this in the following subsection where we extend the three-field model to the multiple-network case. 
%
%
\subsection{Multiple-network systems}\label{sec:model:network} 
The previously introduced three-field formulation can be easily extended to multiple-networks as they are used in certain brain models, see, e.g.,~\cite{VarCTHLTV16}. 
Note, however, that the extension to multiple-networks is also possible for the two-field model. 

We consider the following abstract problem: 
Find~$u\colon[0,T] \to \V$, $y_i\colon[0,T] \to \cZ$, and~$p_i\colon[0,T] \to \Q$ for~$i=1,\dots,\m$, which satisfy the system
\begin{subequations}
\label{eqn:network}
\begin{align}
	a(u,v) - \sum_{i=1}^\m d_i(v, p_i) 
	&= \langle f, v \rangle, \label{eqn:network:a} \\
	(y_i,z)_{\cHZ} - \dd_i(z,p_i) 
	&= 0, \label{eqn:network:b} \\
	d_i(\dot u, q) + c(\dot p_i,q) + \dd_i(y_i,q) - \sum_{j\neq i} \beta_{ij} (p_i-p_j, q)_\Q 
	&= \langle g_i, q\rangle \label{eqn:network:c} 
\end{align}
\end{subequations}
for all test functions~$v\in \V$, $z\in \cZ$, and $q \in \Q$. 
Note that equations~\eqref{eqn:network:b} and~\eqref{eqn:network:c} have to be considered for $i=1,\dots, \m$ and, thus, represent~$\m$ equations each. 
Initial conditions are given by~$p_i(\,\cdot\,,0) = p_i^0 \in \Q$ and define~$u^0\in\V$ as before. 
The assumptions on the spaces and bilinear forms are as in the three-field model of~\Cref{sec:model:three}. This means that each of the bilinear forms~$d_i\colon \V\times \Q \to \R$ and~$\dd_i\colon \cZ\times \Q \to \R$ behave as~$d$ and~$\dd$, respectively. The corresponding continuity constants are denoted by~$C_{d_i}$ and~$C_{\dd_i}$. The presence of additional variables calls for an adjustment of the weak coupling condition.
\begin{assumption}[Weak coupling, network case]
\label{ass:weakCouplingNetwork}
We assume a weak coupling through the bilinear forms~$d_i$ in the sense that
\begin{align*}
  \sum_{i=1}^\m  C_{d_i}^2 \le c_a\, c_c. 
\end{align*}
\end{assumption}
%
Another coupling is described through the parameters~$\beta_{ij}$ in equation~\eqref{eqn:network:c}. 
In the case where they represent exchange rates from one network to the other, it is reasonable to assume symmetry, i.e., $\beta_{ij}=\beta_{ji}$, cf.~\cite{TulV11}. Here, however, we only assume that the coupling parameters are sufficiently small. 
\begin{assumption}[Small exchange rates]
\label{ass:beta}
We assume small exchange rates between the networks, i.e., we assume that 
\[
  \beta :=\max_{i,j \in \{1,\dots, \m\}} |\beta_{ij}|
\]
is small. 
In particular, we assume that~$6\beta\, (\m-1) \le c_c$.
\end{assumption}
\begin{remark}
The introduced network model may be easily extended to the case with different ansatz spaces~$\Q_i$ and~$\cZ_i$ for the variables~$p_i$ and~$y_i$, respectively. 
This may be helpful in order to include varying boundary conditions for the different variables. 
\end{remark}
\begin{remark}[DAE structure, network case]
With the same arguments as in the two-field case, one can show that the semi-discretization of the multiple-network formulation~\eqref{eqn:network} is a DAE of index 1. This follows again by the skew-symmetric block structure and only requires the invertibility of the stiffness and mass matrices. 
\end{remark}
\begin{example}[Poroelastic brain model]\label{exp:brain:network}
In medical applications, multiple-network poroelastic models of the brain with different pressures to distinguish vessel types can be used to investigate cerebral edema, cf.~\cite{VarCTHLTV16}. 
This particular model can be described as an extension of the three-field poroelastic formulation with~$\m=4$. 
Thus, the same spaces and variables as in Example~\ref{exp:poro:threeField} can be used but with multiple pressures $p_i$ and fluid fluxes~$y_i$, $i=1,\dots,4$, in a bounded Lipschitz domain $\Omega \subseteq \R^d$. In this application, no external forces or injections are present ($f\equiv 0$ and $g\equiv 0$) and only hydrostatic pressure gradients drive the system, scaled by very small coupling constants~$\beta_{ij}$, $i,j=1,\dots,4$. 
Thus, \Cref{ass:beta} is satisfied. 
The bilinear forms are chosen as in Example~\ref{exp:poro} except for $d_i$, $\dd_i$, which are defined by 
\begin{align*}
 d_i(v,q) &:= \int_\Omega \alpha_i\, (\nabla \cdot v)\, q \dx, &
 \dd_i(z,q) &:= \int_\Omega \Big(\nabla \cdot \big(\sqrt{\tfrac{\kappa_i}{\nu_i}}\, z\big)\Big)\, q \dx.
\end{align*}
In this medical application, however, inhomogeneous boundary conditions for $p_i$ are considered on two boundary parts (skull and ventricle surface). Thus, additional boundary terms have to be taken into account, see~\cite{VarCTHLTV16}.
\end{example}
%
%
\section{Semi-Explicit Discretization of the Two-Field Model}\label{sec:semiTwoField}
For the numerical solution of~\eqref{eqn:two}, we propose the combination of a {\em semi-explicit} time discretization with step size~$\tau$ and a conforming spatial finite element discretization with mesh size~$h$. 
Thus, we consider a partition of~$[0,T]$ with time points~$t_n=n\,\tau$.
The fully discretized system then reads
\begin{subequations}
\label{eqn:discrete}
\begin{align}
	a(u^{n+1}_h,v_h) - d(v_h, p^{n}_h) &= \langle f^{n+1},v_h\rangle, \label{eqn:discrete:a} \\
	d(D_\tau u^{n+1}_h, q_h) + c(D_\tau p^{n+1}_h,q_h) + b(p^{n+1}_h,q_h) &= \langle g^{n+1},q_h\rangle \label{eqn:discrete:b} 
\end{align}
\end{subequations}
with test functions $v_h\in V_h$ and $q_h \in Q_h$. Hereby, $V_h \subseteq \V$ and $Q_h\subseteq \Q$ denote suitable finite-dimensional spaces resulting from the spatial discretization. The discrete time derivative is denoted by $D_\tau u^{n+1}_h \vcentcolon= (u^{n+1}_h - u^{n}_h) / \tau$ and $u^{n}_h$, $p^{n}_h$ are the resulting approximations of $u(t_{n})$ and $p(t_{n})$, respectively. For the right-hand sides, we introduce $f^{n}\vcentcolon=f(t_{n})$ and~$g^{n}\vcentcolon=g(t_{n})$. For the initial data, we assume~$u^0_h\in V_h$ and~$p^0_h \in Q_h$ to be consistent in the sense of~$a(u^0,v_h) - d(v_h, p^0_h) = \langle f^0,v_h\rangle$.

We emphasize that the scheme~\eqref{eqn:discrete} is semi-explicit in time due to the term~$p^{n}_h$ in equation~\eqref{eqn:discrete:a}. 
The corresponding implicit scheme (with~$p^{n+1}_h$) is analyzed in~\cite{ErnM09}. 
\begin{remark}
The semi-explicit scheme~\eqref{eqn:discrete} may be interpreted as a co-simulation \cite{Mie89,BY11} (also known as waveform relaxation or dynamic iteration) of equations~\eqref{eqn:two:a} and~\eqref{eqn:two:b}.
\end{remark}
%
The advantage of the proposed scheme over a full-implicit discretization is that we can solve sequentially for~$u^{n+1}_h$ by~\eqref{eqn:discrete:a} and afterwards for~$p^{n+1}_h$ by~\eqref{eqn:discrete:b}. Hence, we solve two smaller systems rather than one large system in each time step. Further, the sparsity pattern improves, since both systems only include the solution with standard mass and stiffness matrices. 

Our convergence proof is based on an elliptic-parabolic problem with an additional delay term whose solution~$(\bar u, \bar p)$ only differs by an order of~$\tau$ from the original solution~$(u, p)$. 
The delay system is discussed in the following subsection. 
%
\subsection{A related delay system}
The semi-explicit scheme~\eqref{eqn:discrete} can also be obtained by applying the {\em implicit} Euler method to the delay system
\begin{subequations}
\label{eqn:delay:two}
\begin{align}
	a(\bar{u},v) - d(v, \bar{p}(\,\cdot-\tau)) 
	&= \langle f, v \rangle, \label{eqn:delay:two:a} \\
	d(\dot {\bar{u}}, q) + c(\dot {\bar{p}},q) + b(\bar{p},q) 
	&= \langle g, q\rangle \label{eqn:delay:two:b} 
\end{align}
\end{subequations}
for test functions~$v\in \V$ and~$q \in \Q$.
Note that this changes the nature of the system in the sense that we now need a \emph{history function} for~$\bar{p}$ in $[-\tau, 0]$ rather than only an initial value. Thus, we set~$\bar{p}\big|_{[-\tau, 0]}(t) = \Phi(t)$ and demand
\begin{equation}
\label{eqn:historyTwo}
  \Phi(-\tau) = \Phi(0) = p^0, \qquad  
  \Phi \in C^\infty([-\tau, 0], \Q).
\end{equation}
With this particular history function we have~$\bar p(0) = \Phi(0) = p^0$ and by equation~\eqref{eqn:delay:two:a} we conclude~$\bar u(0) = u^0$, since  
\begin{equation*}
  a(\bar u(0),v)
  = \langle f(0), v \rangle + d(v, \Phi(-\tau)) 
  = \langle f(0), v \rangle + d(v, p^0).
\end{equation*} 
Let us emphasize that the time delay equals the temporal step size $\tau$.
\begin{proposition}
\label{prop:delayTwoField}
Assume sufficiently smooth right-hand sides~$f$ and~$g$ and a history function~$\Phi$ as defined in~\eqref{eqn:historyTwo} such that the solution $(\bar{u},\bar{p})$ of the delay system \eqref{eqn:delay:two} satisfies~${\bar p}\in W^{2,\infty}(\cHQ\!)$. 
Then, the solutions to~\eqref{eqn:two} and~\eqref{eqn:delay:two} are equal up to a term of order~$\tau$. More precisely, we have for almost all~$t\in[0,T]$ that 
\[
  \Vert \bar p(t) - p(t) \Vert^2_\Q
  \lesssim \tau^2\, t\, \big( \Vert \ddot{\Phi}\Vert_{L^\infty(-\tau,0;\cHQ\!)}^2 + \Vert \ddot{\bar p}\Vert_{L^\infty(\cHQ\!)}^2 \big)
\]
and 
\[
  \Vert \bar u(t) - u(t)\Vert^2_\V
  \lesssim \tau^2\, \Vert \dot{\bar p}\Vert^2_{L^2(\cHQ\!)} 
  + \tau^2\, (1+\tau^2)\, t\, \big( \Vert \ddot{\Phi}\Vert_{L^\infty(-\tau,0;\cHQ\!)}^2 + \Vert \ddot{\bar p}\Vert_{L^\infty(\cHQ\!)}^2 \big). 
\]
\end{proposition}
\begin{proof}
We consider the Taylor expansions of $\bar p$ and~$\dot{\bar p}$, 
\begin{align}
\label{eqn:taylor}
\bar p(t-\tau) 
  = \bar p(t) - \tau\, \dot{\bar p}(t) + \tfrac{1}{2} \tau^2\, \ddot{\bar p}(\zeta_t), \qquad
  \dot{\bar p}(t-\tau)   
  = \dot{\bar p}(t) - \tau\, \ddot{\bar p}(\xi_t)
\end{align}
for some~$\zeta_t, \xi_t \in (t-\tau, t) \subseteq (-\tau, T]$. 
With this, the differences~$e_u := \bar u - u$ and~$e_p := \bar p - p$ satisfy the system
\begin{subequations}
\label{eqn:lemDelayTwoField}
\begin{align}
	a(e_u,v) - d(v, e_p) 
	&= - \tau\, d(v, \dot{\bar p}) + \tfrac{1}{2} \tau^2\, d(v, \ddot{\bar p}(\zeta_t)) \label{eqn:lemDelayTwoField:a} \\
	d(\dot e_u, q) + c(\dot e_p,q) + b(e_p,q) 
	&= 0 \label{eqn:lemDelayTwoField:b} 
\end{align}
\end{subequations}
for all test functions~$v\in\V$, $q\in\Q$.  
Note that we have~$e_u(0)=0$ and~$e_p(0)=0$ due to the particular choice of the history function~$\Phi$ in \eqref{eqn:historyTwo}.
On the other hand, considering the derivatives of~\eqref{eqn:two:a} and~\eqref{eqn:delay:two:a} and the Taylor expansion of~$\dot{\bar p}$ gives 
\begin{align}
\label{eqn:lemDelayTwoField:derivative}
  a(\dot e_u,v) - d(v, \dot e_p) 
  = - \tau\, d(v, \ddot{\bar p}(\xi_t)).
\end{align}
Taking the sum of~\eqref{eqn:lemDelayTwoField:derivative} with test function~$v=\dot e_u$ and~\eqref{eqn:lemDelayTwoField:b} with~$q=\dot e_p$, we get with the ellipticity of the bilinear forms and Young's inequality, 
\[
  \Vert \dot e_u(t) \Vert_\V^2 + \Vert \dot e_p(t) \Vert_\cHQ^2 + \tfrac 12 \tddt  \Vert e_p(t) \Vert_b^2
  \lesssim \tau^2\, \Vert \ddot{\bar p}(\xi_t) \Vert_\cHQ^2.
\]
Thus, we conclude by integration over~$[0,t]$ that  
\begin{align}
\label{eqn:inProofDelay}
  \int_0^t \Vert \dot e_u(s)\Vert_\V^2 \ds + \Vert e_p(t) \Vert_\Q^2
  \lesssim \tau^2\,t\, \Vert \ddot{\bar p}\Vert_{L^\infty(-\tau,t;\cHQ\!)}^2.
\end{align}
Besides, the sum of~\eqref{eqn:lemDelayTwoField:a} with~$v=\dot e_u$ and~\eqref{eqn:lemDelayTwoField:b} with~$q=e_p$ yields the estimate 
\begin{align*}
  \tfrac 12\tddt \Vert e_u(t)\Vert_a^2 + \tfrac 12\tddt \Vert e_p(t)\Vert_c^2 + \Vert e_p(t)\Vert_b^2
  &= - \tau\, d(\dot e_u(t), \dot{\bar p}(t)) + \tfrac 12\tau^2\, d(\dot e_u(t), \ddot{\bar p}(\zeta_t)) \\ 
  &\lesssim \Vert \dot e_u(t) \Vert_\V^2
  + \tau^2\, \Vert \dot{\bar p}(t)\Vert_\cHQ^2 
  + \tau^4\, \Vert \ddot{\bar p}(\zeta_t)\Vert_\cHQ^2.  
\end{align*}
Integration over~$[0,t]$ and application of the estimate~\eqref{eqn:inProofDelay} then leads to 
\[
  \Vert e_u (t)\Vert_\V^2
  \lesssim \tau^2\, \Vert \dot{\bar p}\Vert_{L^2(0,t;\cHQ\!)}^2 
  + \tau^2\, (1+\tau^2)\, t\, \Vert \ddot{\bar p}\Vert_{L^\infty(-\tau,t;\cHQ\!)}^2. 
\]
Noting that~$\ddot{\bar p}|_{[-\tau,0]} = \ddot{\Phi}$ finally completes the proof. 
\end{proof}
The previous result states that the solutions~$(u, p)$ and~$(\bar u, \bar p)$ are close as long as the solution of the related delay system stays stable. 
We would like to point out that the stability of delay PDAEs is a delicate topic, particularly in the current setting where the system structure is of \emph{neutral type}. 
This means that~$\dot{\bar{p}}$ at time~$t$ depends on~$\dot{\bar{p}}(t-\tau)$. 
From the PDE side one can expect difficulties, since even delay systems of \emph{retarded type}, where the differential equation only includes a delay of the form~${\bar{p}}(t-\tau)$, may not gain smoothness in contrast to the finite-dimensional setting~\cite{AltZ18}. 
On the other hand, a DAE with a delay term in the constraint may even behave like an \emph{advanced} equation \cite{Cam80}, i.e., the solution at time $t$ depends on the derivative of the solution at time $t-\tau$. As a direct consequence one can only expect solutions in a distributional setting \cite{TreU19}. Classical solutions may be obtained, if a certain structure is imposed on the delay DAE and the history function satisfies so-called \emph{splicing conditions} \cite{Ung18}. 
For our particular case, we can prove that~$\ddot{\bar p}$ indeed stays uniformly bounded for smooth data, see \Cref{sec:neutralDelayPDE} for further details.

For an infinite time interval, i.e., $T=\infty$, such a stability result is only possible under a weak coupling condition in the spirit of~\Cref{ass:weakCouplingTwo}. We illustrate this in the following finite-dimensional example. 
\begin{example}
	\label{ex:spectralRadius}
A spatial discretization of the delay PDAE \eqref{eqn:delay:two} can be written as
\begin{subequations}
\label{eqn:two:DDAE}
\begin{align}
	K_a \bar{u}_h(t) - D^T \bar{p}_h(t-\tau) &= f_h(t),\\
	D\dot{\bar{u}}_h(t) + M_c\dot{\bar{p}}_h(t) + K_b \bar{p}_h(t) &= g_h(t),
\end{align}
\end{subequations}
with the matrices from \Cref{rem:two:DAE}.
Solving the first equation for $\bar{u}_h$ and substituting in the second equation results in the neutral delay differential equation
\begin{equation}
	\label{eqn:twoDelay:neutral}
	M_c\dot{\bar{p}}_h(t) + K_b \bar{p}_h(t) = -DK_a^{-1}D^T\dot{\bar{p}}_h(t-\tau) - K_a^{-1}\dot{f}_h(t) + g_h(t).
\end{equation}
A necessary condition for the \emph{asymptotic stability} of \eqref{eqn:twoDelay:neutral} is that the spectral radius of $M_c^{-1}DK_a^{-1}D^T$ is strictly smaller than one, see \cite[Th.~3.20]{GuKC03} for further details. Note that this condition quantitatively resembles the weak coupling condition in  \Cref{ass:weakCouplingTwo}. 
\end{example}
We emphasize that even if~$\ddot{\bar p}$ stays bounded for finite times~$T$, it may become very large. 
A $T$-independent bound requires again a weak coupling condition, which is also observable numerically, cf.~\Cref{sec:num:weakCoupling}. 
%
%
\subsection{Spatial projection}
Based on the two-field formulation \eqref{eqn:two} and discrete spaces $V_h$ and $Q_h$, we define the projections $\Ru \colon \V\to  V_h$ and $\Rp \colon \Q \to Q_h$ by
\begin{align}
\label{eqn:projTwoFieldU}	
	a(\Ru u,v_h) 
	= a(u,v_h)
\end{align}
for all $v_h \in V_h$ and
\begin{align}
\label{eqn:projTwoFieldP}
	b(\Rp p,q_h) 
	= b(p,q_h)
\end{align}
for all $q_{h} \in Q_{h}$. 
Note that~$\Ru$ and $\Rp$ are well-defined due to the ellipticity of~$a$ and~$b$.
For the following error analysis, we need certain approximation properties of the projectors. 
\begin{assumption}[Spatial projection]
\label{ass:projErrorTwoField}
Consider $u \in \V$ and $p \in \Q$. We assume that the projection errors satisfy
\begin{align*}
  \|u - \Ru u\|_\cHV &\lesssim h\,\|u\|_\V, \qquad\quad
  \|u - \Ru u\|_\V \lesssim h\,\|\nabla^2 u\|_\cHV\\
  \|u - \Rp p\|_\cHQ &\lesssim h\,\|p\|_\Q, \qquad\quad 
  \|p - \Rp p\|_\Q \lesssim h\,\|\nabla^2 p\|_\cHQ, 
\end{align*}
if the second derivatives $\nabla^2 u$ and $\nabla^2 p$ are bounded in $\cHV$ and $\cHQ$, respectively. 
\end{assumption}
\begin{example}
\label{ex:P1Two}
For the spaces~$\V = [H^1_0(\Omega)]^d$, $\cHV = [L^2(\Omega)]^d$, $\Q = H^1_0(\Omega)$, and $\cHQ = L^2(\Omega)$, \Cref{ass:projErrorTwoField} is satisfied if~$V_h$ and~$Q_h$ equal the standard~$P_1$ Lagrange finite element spaces, see e.g.~\cite[Ch.~II.6-II.7]{Bra07} for more details.
\end{example}
%
%
\subsection{Full discretization of the delay system}
\label{subsec:twoFieldDelayConvergence}
As mentioned above, we prove the convergence of the semi-explicit scheme~\eqref{eqn:discrete} by the interpretation as an implicit discretization of the delay system~\eqref{eqn:delay:two}. 
The following proposition quantifies the error estimate between the fully discrete solution and the exact solution to the delay system. 
\begin{proposition}
\label{prop:discErrorTwoField}
Suppose \Cref{ass:weakCouplingTwo,ass:projErrorTwoField} and the assumptions of~\Cref{prop:delayTwoField} hold, as well as $\nabla^2 \bar u \in L^\infty(\cHV)$ and~$\nabla^2 \bar p\in L^\infty(\cHQ)$. 
Then, taking initial data~$u_h^0\in V_h$ and~ $p_h^0\in Q_h$ with 
\[
  \|\Ru u^0 - u_h^0\|_\V + \|\Rp p^0 - p^0_h\|_\cHQ 
  \lesssim\, h
\]
implies that for all $n \le T/\tau$ the solution of the fully discretized system~\eqref{eqn:discrete} satisfies 
\begin{equation*}
\label{eq:discErrorTwo}
  \|\bar u(t_n) - u^n_h \|^2_{\V} + \|\bar p(t_n) - p^n_h \|^2_{\cHQ} + 
  \sum_{k=1}^{n} \tau\, \|\bar p(t_k) - p^k_h \|_{\Q}^2 
  \ \lesssim\ \mathrm{e}^{t_n}(1+t_n)\, ( h^2 + \tau^2).
\end{equation*}
\end{proposition}
Before we prove~\Cref{prop:discErrorTwoField}, we state the following useful lemma, which is easily proven by straight-forward calculations.
\begin{lemma}
\label{lem:symBilinear}
For a symmetric bilinear form $\mathfrak{a}$ it holds that
\begin{align*}
	2 \mathfrak{a}(u, u-v)
	= \Vert u \Vert^2_\mathfrak{a} - \Vert v \Vert^2_\mathfrak{a} + \Vert u- v \Vert^2_\mathfrak{a}
\end{align*}
with $\Vert \cdot \Vert_\mathfrak{a}^2 := \mathfrak{a}(\,\cdot\,,\cdot\,)$.
\end{lemma}
\begin{proof}[Proof of \Cref{prop:discErrorTwoField}] 
We follow the ideas presented in \cite{ErnM09} and introduce
\begin{displaymath}
	\eta_u^n \vcentcolon= \Ru \bar u^n - u^n_h \in V_h\qquad\text{and}\qquad
	\eta_p^n \vcentcolon= \Rp \bar  p^n - p^n_h \in Q_h, 
\end{displaymath}
where $\bar u^{n}\vcentcolon=\bar u(t_{n})$ and $\bar p^{n} \vcentcolon=\bar p(t_{n})$ are the solutions of \eqref{eqn:delay:two} and $\Ru$, $\Rp$ denote the projections defined in \eqref{eqn:projTwoFieldU} and \eqref{eqn:projTwoFieldP}, respectively. Note, however, that these projections differ from the projections used in~\cite{ErnM09}. 
Using~\eqref{eqn:discrete:a} and~\eqref{eqn:delay:two:a}, we immediately obtain
\begin{align*}
	a(\eta^{n+1}_u,v_h) - d(v_h,\eta^{n+1}_p) 
	&= a(\bar u^{n+1}-u^{n+1}_h,v_h) - d(v_h,\Rp \bar p^{n}-p^{n}_h) - d(v_h,\eta^{n+1}_p- \eta^{n}_p)\\
	&= d(v_h,\bar p^{n}- \Rp \bar p^{n}) - d(v_h,\eta^{n+1}_p- \eta^{n}_p)
\end{align*}
for all test functions $v_h\in V_h$. Similarly, we observe that
\begin{align*}
	\tau\, b(\eta^{n+1}_p,q_h) 
	&= \tau\, b(\bar p^{n+1} - p^{n+1}_h,q_h) \\
	&= -d(\tau \dot{\bar u}^{n+1},q_h) - c(\tau \dot{\bar p}^{n+1},q_h) + d(\tau D_\tau u^{n+1}_h,q_h) + c(\tau D_\tau p^{n+1}_h,q_h)
\end{align*}
for all $q_h \in Q_h$. Together with 
\begin{displaymath}
	\theta^{n+1}_u \vcentcolon= \Ru \bar u^{n+1} - \Ru \bar u^n - \tau \dot{\bar u}^{n+1} \qquad\text{and}\qquad
	\theta^{n+1}_p \vcentcolon= \Rp \bar p^{n+1} -\Rp \bar p^n -\tau \dot{\bar p}^{n+1}, 
\end{displaymath}
this implies 
\begin{align*}
	&d(\eta^{n+1}_u - \eta^{n}_u,q_h) + c(\eta^{n+1}_p - \eta^n_p,q_h) + \tau\, b(\eta^{n+1}_p,q_h) \\
	&\ =  d(\Ru \bar u^{n+1} - \Ru \bar u^n - \tau D_\tau u^{n+1}_h,q_h) + c(\Rp \bar p^{n+1} - \Rp \bar p^n - \tau D_\tau p^{n+1}_h,q_h) + \tau\, b(\eta^{n+1}_p,q_h)\\
	&\ = d(\theta^{n+1}_u,q_h) + c(\theta^{n+1}_p,q_h)
\end{align*}
for all $q_h\in \Q_h$. 
For the particular choices~$v_h = \eta^{n+1}_u - \eta^n_u$ and $q_h = \eta^{n+1}_p$, we obtain
\begin{multline*}
	a(\eta^{n+1}_u,\eta^{n+1}_u - \eta^n_u) + c(\eta^{n+1}_p-\eta^n_p,\eta^{n+1}_p) + \tau\, b(\eta^{n+1}_p,\eta^{n+1}_p)\\ 
	= d(\eta^{n+1}_u - \eta^n_u,\bar p^{n}- \Rp \bar p^{n} - \eta^{n+1}_p + \eta^{n}_p) + d(\theta^{n+1}_u,\eta^{n+1}_p) + c(\theta^{n+1}_p,\eta^{n+1}_p).
\end{multline*}
Using \Cref{lem:symBilinear} for the bilinear forms $a$ and $c$, we obtain
\begin{multline*}
	\|\eta^{n+1}_u\|^2_a - \|\eta^n_u\|^2_a + \|\tau D_\tau\eta^{n+1}_u\|^2_a + \|\eta^{n+1}_p\|^2_c - \|\eta^n_p\|^2_c + \|\tau D_\tau\eta^{n+1}_p\|^2_c + 2\tau\, \|\eta^{n+1}_p\|^2_b \\
	= 2\, d(\tau D_\tau\eta^{n+1}_u,\bar p^{n}- \Rp \bar p^{n} - \tau D_\tau\eta^{n+1}_p) + 2\, d(\theta^{n+1}_u,\eta^{n+1}_p) + 2\, c(\theta^{n+1}_p,\eta^{n+1}_p).
\end{multline*}
With the identity
\begin{multline*}
d(\tau D \eta^{n+1}_u,\bar p^{n}- \Rp\bar p^{n}) \\
= d(\eta_u^{n+1},\bar p^{n}-\Rp \bar p^n) - d(\eta_u^n,\bar p^{n-1}-\Rp \bar p^{n-1}) - d(\eta_u^n,(\bar p^n-\bar p^{n-1}) - \Rp(\bar p^n - \bar p^{n-1}))
\end{multline*}
and the Taylor expansion~$\bar p^{n} - \bar p^{n-1} = \dot{\bar p}(\xi)$, the weighted version of Young's inequality, cf.~\cite[App.~B]{Eva98}, leads to
\begin{multline*}
2\, d(\tau D \eta^{n+1}_u,\bar p^{n}- \Rp\bar p^{n} - \tau D\eta^{n+1}_{p}) \\
\begin{aligned}
&\leq \tfrac{C_{d}^2}{c_a\,c_c}\|\tau D \eta^{n+1}_u\|_a^2 + \|\tau D\eta^{n+1}_{p}\|_c^2 + \tau \|\eta_u^n\|_a^2 + \tau \tfrac{C_{d}^2}{c_a} \|\dot{\bar p} - \Rp \dot{\bar p}\|_{L^\infty(\cHQ)}^2 \\
&\phantom{\leq}\quad + 2\, d(\eta_u^{n+1},\bar p^{n}-\Rp \bar p^{n}) 
- 2\, d(\eta_u^{n},\bar p^{n-1}-\Rp \bar p^{n-1}).
\end{aligned}
\end{multline*}
Similarly, we obtain for the two other terms 
\begin{align*}
	2\, d(\theta^{n+1}_u,\eta^{n+1}_p) 
	\leq \tfrac{2\, \tilde C_d}{\sqrt{c_b}}\|\theta^{n+1}_u\|_{\cHV} \,\|\eta^{n+1}_p\|_{b}
	\leq \tfrac{\tilde C_d^2}{c_b}\tfrac{2}{\tau}\|\theta^{n+1}_u\|_{\cHV}^2 + \tfrac{\tau}{2} \|\eta^{n+1}_p\|^2_{b}
\end{align*}
and with the continuity constant~$\CQtoH$ of the embedding~$\Q\hook \cHQ$, 
\begin{align*}
	2\, c(\theta^{n+1}_p,\eta^{n+1}_p) &\leq 2\,C_c \|\theta^{n+1}_p\|_\cHQ \, \|\eta^{n+1}_p\|_\cHQ \leq \tfrac{C_c^2\, \CQtoHsquare}{c_b}\tfrac{2}{\tau} \|\theta^{n+1}_p\|^2_\cHQ + \tfrac{\tau}{2} \|\eta^{n+1}_p\|^2_b. 
\end{align*}
Next, we combine the previous estimates and absorb the terms $\|\tau D_\tau\eta^{n+1}_p\|_c^2$, $ \tau \|\eta^{n+1}_p\|_b$, and $\|\tau D_\tau\eta^{n+1}_u\|^2_a$ using \Cref{ass:weakCouplingTwo} for the latter. 
Invoking~\Cref{ass:projErrorTwoField} then yields
\begin{multline}
\label{eqn:estimateEtaAtN}
	\|\eta^{n+1}_u\|^2_a - (1+\tau) \|\eta^n_u\|^2_a + \|\eta^{n+1}_p\|^2_c - \|\eta^n_p\|^2_c + \tau\, \| \eta^{n+1}_p\|^2_b \\
	-2\, d(\eta_u^{n+1},\bar p^{n}-\Rp \bar p^{n}) + 2\, d(\eta_u^{n},\bar p^{n-1}-\Rp \bar p^{n-1})\\
	\begin{aligned}
	&\lesssim \tau h^2 \|\dot{\bar p}\|_{L^\infty(\Q)}^2 + \tfrac 1\tau \|\theta^{n+1}_u\|^2_\cHV + \tfrac1\tau \|\theta^{n+1}_p\|^2_\cHQ.
	\end{aligned}
\end{multline}
To estimate~$\theta^{n+1}_u$ and~$\theta^{n+1}_p$, we observe 
\begin{align*}
	\theta^{n+1}_u 
	&= \int_{t_n}^{t_{n+1}} \Ru \dot{\bar u}(s) \ds - \Big( (s-t_n)\dot{\bar u}(s)\Big|_{t_n}^{t_{n+1}} - \int_{t_n}^{t_{n+1}} \dot{\bar u}(s)\ds\Big) - \int_{t_n}^{t_{n+1}} \dot{\bar u}(s)\ds\\
	&= -\int_{t_n}^{t_{n+1}} \dot{\bar u}(s)-\Ru\dot{\bar u}(s)\ds - \int_{t_n}^{t_{n+1}} (s-t_n)\, \ddot{\bar u}(s)\ds
\end{align*}
and thus, by~\Cref{ass:projErrorTwoField}, 
\begin{align*}
	\|\theta^{n+1}_u\|_{\cHV} 
	\leq \int_{t_n}^{t_{n+1}} \|\dot{\bar u}-\Ru \dot{\bar u}\|_{\cHV} \ds + \tau^2\|\ddot{\bar u}\|_{L^\infty(t_n,t_{n+1};\cHV)}
	\lesssim \tau h\, \|\dot{\bar u}\|_{L^\infty(\V)} + \tau^2\|\ddot{\bar u}\|_{L^\infty(\cHV)}.
\end{align*}
Note that the regularity of the history function~$\Phi$ and ${\bar p}\in W^{2,\infty}(\cHQ)$ imply~${\bar u}\in W^{2,\infty}(\V)$ by~\eqref{eqn:delay:two:a}.
In the same manner, we obtain for~$\theta^{n+1}_p$ the estimate 
\begin{equation*}
\|\theta^{n+1}_p\|_\cHQ 
\lesssim \tau h\, \|\dot{\bar p}\|_{L^\infty(\Q)} + \tau^2\|\ddot{\bar p}\|_{L^\infty(\cHQ\!)}.
\end{equation*} 
Taking the sum over $n$ in~\eqref{eqn:estimateEtaAtN}, we finally obtain 
\begin{align*}
	\|\eta^{n}_u &\|^2_a + \|\eta^{n}_p  \|^2_c 
	+ \sum_{k=1}^{n} \tau\, \|\eta^{k}_p\|^2_b\\
%
%
	&\lesssim \mathrm{e}^{t_n}\Big[h^2  
	+ t_{n}\, \Big( h^2 \|\dot{\bar p}\|^2_{L^\infty(\Q)}
 	+ h^2\, \|\dot{\bar u}\|^2_{L^\infty(\V)} + \tau^2\, \|\ddot{\bar u}\|^2_{L^\infty(\cHV)} \\
	&\hspace{5.5cm} + h^2\, \|\dot{\bar p}\|^2_{L^\infty(\Q)} + \tau^2\, \|\ddot{\bar p}\|^2_{L^\infty(\cHQ\!)} \Big) + h^2 \|{\bar p}\|^2_{L^\infty(\Q)}\Big] \\
	&\lesssim \mathrm{e}^{t_n}(1+t_n)\, (h^2+\tau^2). 
\end{align*}
Note that the exponential factor appears due to the `perturbed' telescope sum in~\eqref{eqn:estimateEtaAtN} and the application of a discrete Grönwall inequality. 
Finally, using the assumed regularity and~\Cref{ass:projErrorTwoField}, i.e., 
\[
\|\bar u^{n} - \Ru \bar u^{n}\|_\V 
\lesssim h\, \|\nabla^2 \bar u^n\|_{\cHV}, \
\|\bar p^{n} - \Rp \bar p^{n} \|_\cHQ\! 
\lesssim h\, \|\bar p^n\|_{\Q}, \
\|\bar p^{n} - \Rp \bar p^{n}\|_\Q
\lesssim h\, \|\nabla^2\bar p^n\|_{\cHQ\!},
\]
the assertion follows by the triangle inequality.
\end{proof}
\begin{remark}
With the same assumptions as in \Cref{prop:discErrorTwoField} and a slightly stronger coupling condition (namely with a factor~$1+\eps$), one can also show that the considered error is bounded by a constant times~$t_{n}\, (h^2+\tau^2 + h^4\tau^{-1})$. 
Thus, the exponential term can be exchanged by a higher-order term, which includes the step size~$\tau$ in the denominator. This, however, is not critical in the range of interest with~$\tau \approx h$.
\end{remark}
%
%
\subsection{Convergence of the semi-explicit scheme}
We close this section with a summary of the previous results, which states that the semi-explicit scheme~\eqref{eqn:discrete} converges with order~$h+\tau$ if the finite element spaces are chosen appropriately and the weak coupling condition is satisfied. 
\begin{theorem}[Convergence, two-field model]
\label{thm:twoFieldConvergence}
Suppose that \Cref{ass:weakCouplingTwo,ass:projErrorTwoField} hold. Further, let the right-hand sides~$f\colon [0,T] \to \cHV$ and~$g\colon [0,T] \to \cHQ$ be sufficiently smooth.  
Then, with~$(u, p)$ being the solution of the original system~\eqref{eqn:two} and $u^n_h\in V_h$, $p^n_h\in Q_h$ the fully discrete approximations obtained by~\eqref{eqn:discrete} for $n \le T/\tau$ and initial data~$u_h^0\in V_h$, $p_h^0\in Q_h$ with
\[
\|\Ru u^0 - u_h^0\|_\V + \|\Rp p^0 - p^0_h\|_\cHQ 
\lesssim\, h,
\]
we obtain the error estimate 
\begin{equation*}
\|u(t_n) - u_h^n\|^2_\V + \|p(t_n) - p_h^n\|^2_\cHQ + \sum_{k=1}^n\tau\, \|p(t_k) - p_h^k\|^2_\Q \ \lesssim\ \mathrm{e}^{t_n} (1+t_n)\, ( h^2 + \tau^2).
\end{equation*}
\end{theorem}
\begin{proof}	
If we define the history function~$\Phi$ as in~\eqref{eqn:historyTwo}, then the assumptions on the data imply that the solution of the related delay system~\eqref{eqn:delay:two} stays bounded in the sense of~${\bar p}\in W^{2,\infty}(\cHQ\!)$, cf.~\Cref{sec:neutralDelayPDE}. 
Further, the assumed~$H^2$-regularity of the bilinear forms~$a$ and~$b$ yields that~$\nabla^2 \bar u$ and~$\nabla^2 \bar p$ are bounded as well, i.e., $\nabla^2 \bar u \in L^\infty(\cHV)$ and~$\nabla^2 \bar p\in L^\infty(\cHQ\!)$. 
Thus, all assumptions of~\Cref{prop:delayTwoField,prop:discErrorTwoField} are satisfied and the stated estimate directly follows from the previous results and the triangle inequality. \Cref{prop:delayTwoField} shows that the continuous solutions~$(u, p)$ and~$(\bar u, \bar p)$ are close, whereas~\Cref{prop:discErrorTwoField} shows that the fully discrete solution approximates the solution of the delay system with the given order. 
\end{proof}
%
%
\section{Semi-Explicit Discretization of the Network Model}\label{sec:semiNetwork}
Similar to the two-field model discussed in~\Cref{sec:semiTwoField}, we now consider a semi-explicit time discretization of the multiple-network system~\eqref{eqn:network}. Note that this includes the three-field formulation as a special case for~$\m=1$. 
For the network model, the combination of semi-explicit time discretization and conforming spatial discretization leads to 
\begin{subequations}
\label{eqn:discrete:network}
\begin{align}
	a(u^{n+1}_h,v_h) - \sum_{i=1}^\m d_i(v_h, p^n_{i,h}) 
	&= \langle f^{n+1}, v_h \rangle, \label{eqn:discrete:network:a} \\
	(y^{n+1}_{i,h},z_h)_{\cHZ} - \dd_i(z_h,p^{n+1}_{i,h}) 
	&= 0, \label{eqn:discrete:network:b} \\
	d_i(D_\tau u^{n+1}_h, q_h) + c(D_\tau p^{n+1}_{i,h},q_h) + \dd_i(y^{n+1}_{i,h},q_h) - \sum_{j\neq i} \beta_{ij} (p^{n+1}_{i,h}-p^{n+1}_{j,h}, q_h)_\Q 
	&= \langle g^{n+1}_i, q_h\rangle \label{eqn:discrete:network:c} 
\end{align}
\end{subequations}
for all test functions~$v_h\in V_h$, $z_h\in Z_h$, $q_h \in Q_h$ and~$i=1,\dots,\m$. 
As before, we consider a partition of~$[0,T]$ with time points~$t_n=n\,\tau$, conforming finite element spaces~$V_h \subseteq \V$, $Z_h \subseteq \cZ$, $Q_h\subseteq \Q$, and $u^{n}_h$, $y^{n}_{i,h}$, $p^{n}_{i,h}$ denote the fully-discrete approximations at time~$t_{n}$. 

Throughout this section, we consider the weak coupling condition from~\Cref{ass:weakCouplingNetwork} as well as the small exchange condition from~\Cref{ass:beta}.
%
\subsection{A related network model with delay}
We insert a delay term to system~\eqref{eqn:network}, i.e., we consider the solution~$(\bar u, \bar y_i, \bar p_i)$ to
\begin{subequations}
\label{eqn:delay:network}
\begin{align}
	a(\bar u,v) - \sum_{i=1}^\m d_i(v, \bar p_i(\,\cdot-\tau)) 
	&= \langle f, v \rangle, \label{eqn:delay:network:a} \\
	(\bar y_i,z)_{\cHZ} - \dd_i(z,\bar p_i) 
	&= 0, \label{eqn:delay:network:b} \\
	d_i(\dot {\bar u}, q) + c(\dot {\bar p}_i,q) + \dd_i(\bar y_i,q) - \sum_{j\neq i} \beta_{ij} (\bar p_i-\bar p_j, q)_\Q 
	&= \langle g_i, q\rangle \label{eqn:delay:network:c} 
\end{align}
\end{subequations}
for~$i=1,\dots,\m$ and all test functions~$v\in \V$, $z\in \cZ$, and $q \in \Q$. 
Here, we need~$\m$ history functions for~$\bar p_i|_{[-\tau, 0]}(t) = \Phi_i(t)$ and set  
\begin{align}
\label{eqn:historyNetwork}
\Phi_i(-\tau) = \Phi_i(0) = p_i^0, \qquad  
\Phi_i \in C^\infty([-\tau, 0], \Q).
\end{align}
This then implies~$\bar p_i(0) = p_i(0)$ and~$\bar u(0) = u(0)$. 
As for the two-field formulation, we compare the solutions of the original and the delay system. 
\begin{proposition}
\label{prop:delayNetwork}
Assume sufficiently smooth right-hand sides~$f$, $g_i$ and history functions~$\Phi_i$ as defined in~\eqref{eqn:historyNetwork} such that the solution of~\eqref{eqn:delay:network} satisfies~${\bar p}_i\in W^{2,\infty}(\Q)$ for all~$i=1,\dots, \m$. 
Then, the difference of the solutions to~\eqref{eqn:network} and~\eqref{eqn:delay:network} satisfy the estimate
\begin{align*}
\Vert \bar u(t) - u(t) \Vert_\V^2
+ \sum_{i=1}^\m \Vert\bar p_i(t) - p_i(t) \Vert_\Q^2 
+ \sum_{i=1}^\m \Vert \bar y_i - y_i\Vert^2_{L^2(\cHZ)} 
\lesssim \tau^2\, \m\, \big( 1 + \mathrm{e}^{C (1 +4 \m^2\beta^2)\, t} \big)  \bar P
\end{align*}
with a constant~$C$ independent of~$\tau$ and 
\[
  \bar P
  := \sum_{i=1}^\m \Vert \dot{\bar p}_i \Vert_{L^2(\Q)}^2 
  + (1 + \tau^2)\, T\, \sum_{i=1}^\m \Big( \Vert \ddot{\Phi}_i\Vert_{L^\infty(-\tau,0;\Q)}^2 + \Vert \ddot{\bar p}_i\Vert_{L^\infty(\Q)}^2 \Big).
\]
%
\end{proposition}
\begin{proof}
We introduce the error terms~$e_u := \bar u - u$, $e_{y_i} := \bar y_i - y_i$, $e_{p_i} := \bar p_i - p_i$ and note that~$e_u(0)=0$ and~$e_{p_i}(0)=0$ by construction of the delay system~\eqref{eqn:delay:network}. 
Using a Taylor expansion as in~\eqref{eqn:taylor} for each~$p_i$ with some~$\zeta_{i,t}, \xi_{i,t} \in (t-\tau, t) \subseteq (-\tau, T]$, we obtain the system 
\begin{subequations}
\label{eqn:lemDelayNetwork}
\begin{align}
	a(e_u,v) - \sum_{i=1}^\m d_i(v, e_{p_i}) 
	&= - \tau\, \sum_{i=1}^\m d_i(v, \dot{\bar p}_i) + \tfrac{1}{2}\,\tau^2 \sum_{i=1}^\m d_i(v, \ddot{\bar p}_i(\zeta_{i,t})), \label{eqn:lemDelayNetwork:a} \\
	(e_{y_i},z)_{\cHZ} - \dd_i(z,e_{p_i}) 
	&= 0, \label{eqn:lemDelayNetwork:b} \\
	d_i(\dot e_u, q) + c(\dot e_{p_i},q) + \dd_i(e_{y_i},q) 
	&= \sum_{j\neq i} \beta_{ij} (e_{p_i} - e_{p_j}, q)_{\Q} \label{eqn:lemDelayNetwork:c}
\end{align}
\end{subequations}
for test functions~$v\in\V$, $z\in\cZ$, and~$q\in\Q$. 
From equation~\eqref{eqn:lemDelayNetwork:b} and~$e_{p_i}(0)=0$ we conclude that also~$e_{y_i}(0)=0$ for all~$i=1,\dots,\m$. 
Considering the derivatives of the first two equations, we get 
\begin{subequations}
\begin{eqnarray}
	a(\dot e_u,v) - \sum_{i=1}^\m d_i(v, \dot e_{p_i}) &=& -\tau\, \sum_{i=1}^\m d_i(v, \ddot{\bar p}_i(\xi_{i,t})), \label{eqn:lemDelayNetwork:derivative:a} \\
	(\dot e_{y_i},z)_{\cHZ} - \dd_i(z,\dot e_{p_i}) &=& 0.  \label{eqn:lemDelayNetwork:derivative:b} 
\end{eqnarray}
\end{subequations} 
The sum of~\eqref{eqn:lemDelayNetwork:derivative:a} with~$v=\dot e_u$, \eqref{eqn:lemDelayNetwork:derivative:b} with~$z=e_{y_i}$, and~\eqref{eqn:lemDelayNetwork:c} with~$q=\dot e_{p_i}$ for all~$i=1,\dots, \m$ leads to 
\begin{align*}
  \Vert \dot e_u \Vert_\V^2 
  + \sum_{i=1}^\m \Vert \dot e_{p_i}\Vert_\Q^2 
  + \tddt \sum_{i=1}^\m \Vert e_{y_i}\Vert^2_{\cHZ} 
  \lesssim \m\, \beta^2 \sum_{i=1}^\m \sum_{j\neq i} \Vert e_{p_i} - e_{p_j}\Vert_\Q^2
  + \m\,\tau^2 \sum_{i=1}^\m \Vert \ddot{\bar p}_i(\xi_{i,t}) \Vert_\Q^2.
\end{align*} 
Integration over~$[0,t]$ then yields 
\begin{align*}
  \int_0^t \Vert \dot e_u(s) \Vert_\V^2  \ds 
  &\lesssim \m\,\beta^2 \sum_{i=1}^\m \sum_{j\neq i} \int_0^t \Vert e_{p_i}(s) - e_{p_j}(s)\Vert_\Q^2 \ds
  + \m\,\tau^2\,t\, \sum_{i=1}^\m \Vert \ddot{\bar p}_i\Vert_{L^\infty(-\tau,t;\Q)}^2.
\end{align*} 
On the other hand, the sum of~\eqref{eqn:lemDelayNetwork:a} with~$v=\dot e_u$, \eqref{eqn:lemDelayNetwork:b} with~$z=e_{y_i}$, and~\eqref{eqn:lemDelayNetwork:c} with~$q= e_{p_i}$ for all~$i=1,\dots, \m$ gives 
\begin{align*}
  &\tfrac 12 \tddt \Vert e_u \Vert_\V^2
   + \tfrac 12 \tddt \sum_{i=1}^\m \Vert e_{p_i} \Vert_\Q^2  
  + \sum_{i=1}^\m \Vert e_{y_i} \Vert^2_{\cHZ} \\
  &\,\lesssim \m\,\beta^2 \sum_{i=1}^\m\sum_{j\neq i} \Vert e_{p_i} - e_{p_j}\Vert^2_\Q 
  + \sum_{i=1}^\m \Vert e_{p_i} \Vert^2_\Q
  + \Vert \dot e_u\Vert_\V^2 + \m\,\tau^2 \sum_{i=1}^\m \Vert \dot{\bar p}_i \Vert_\Q^2 
  + \m\,\tau^4 \sum_{i=1}^\m \Vert \ddot{\bar p}_i(\zeta_{i,t}) \Vert_\Q^2.
\end{align*}
We integrate again over~$[0,t]$ and use the previous estimate of the integral of~$\Vert \dot e_u\Vert_\V^2$. With the triangle inequality applied to~$\Vert e_{p_i}(s)-e_{p_j}(s) \Vert^2_\Q$ this then yields  
\begin{align*}
  &\Vert e_u(t) \Vert_\V^2
  + \sum_{i=1}^\m \Vert e_{p_i}(t) \Vert_\Q^2
  + \sum_{i=1}^\m \int_0^t \Vert e_{y_i}(s) \Vert^2_{\cHZ} \ds \\
  &\,\lesssim (1 +4 \m^2\beta^2 ) \sum_{i=1}^\m \int_0^t \Vert e_{p_i}(s) \Vert^2_\Q \ds
  + \m\,\tau^2 \sum_{i=1}^\m \Vert \dot{\bar p}_i \Vert_{L^2(0,t;\Q)}^2 
  + \m\,(\tau^2 + \tau^4)\, t\, \sum_{i=1}^\m \Vert \ddot{\bar p}_i\Vert_{L^\infty(-\tau,t;\Q)}^2.
\end{align*}
Finally, an application of Grönwall's inequality provides 
\begin{align*}
  \sum_{i=1}^\m \Vert e_{p_i}(t) \Vert_\Q^2 
  \le C\,\m\, \mathrm{e}^{C (1 +4 \m^2\beta^2)\, t} \Big[ \tau^2 \sum_{i=1}^\m \Vert \dot{\bar p}_i \Vert_{L^2(0,t;\Q)}^2 
  + (\tau^2 + \tau^4)\, t\, \sum_{i=1}^\m \Vert \ddot{\bar p}_i\Vert_{L^\infty(-\tau,t;\Q)}^2 \Big],
\end{align*}
where~$C$ is the constant hidden in~$\lesssim$ of the previous estimate. Note that this constant is independent of the discretization parameter~$\tau$. 
\end{proof}
The latter result shows that the solutions of the original network model~\eqref{eqn:network} and the corresponding delay model~\eqref{eqn:delay:network} only differ by a term of order~$\tau$ as long as~$\bar P(t)$ stays bounded, i.e., as long as the delay system has a stable solution. 
Recall that this stability issue is discussed in~\Cref{sec:neutralDelayPDE} for the two-field model. In the setting of smooth data and regular solutions considered there, the two- and three-field formulation are equivalent. 
In the network case, the operators turn into operator matrices with similar properties. The only difference is that the ellipticity of the differential operator becomes a G\aa rding inequality. This, however, does not effect the stability result. 

We move on with the discretization of the delay system, which defines the semi-explicit scheme introduced in~\eqref{eqn:discrete:network}.
%
%
\subsection{Spatial projection}
Let $u\in \V$, $y_i\in\cZ$, and~$p_i\in\Q$, $i=1,\dots,\m$. Based on the network system \eqref{eqn:network} and the discrete spaces $V_h$, $Z_h$, and $Q_h$, we define two projection operators. 
As in the two-field model we define~$\Ru\colon \V\to V_h$ by 
\begin{align}
\label{eqn:projThreeFieldU}
  a(\Ru u,v_h) = a(u,v_h)  
\end{align}
for all $v_h \in V_h$. Note that this problem is uniquely solvable due to the ellipticity of $a$. 
Second, we define the coupled projection~$\Ri\colon \cZ\times \Q\to Z_h\times Q_h$ and write in short~$\Ri y_i$ and~$\Ri p_i$ for the parts of~$\Ri(y_i,p_i)$ in~$Z_h$ and~$Q_h$, respectively. Using this notation, we define
\begin{subequations}
\label{eqn:projThreeField}
\begin{align}
  (\Ri y_i,z_{h})_{\cHZ} - \dd_i(z_{h},\Ri p_i) 
  &= \phantom{- \dd_i}(y_i,z_{h})_{\cHZ} - \dd_i(z_{h}, p_i), \label{eqn:projThreeFieldY} \\
  - \dd_i(\Ri y_i,q_{h}) \hspace{2.86cm}
  &= - \dd_i(y_i,q_{h})  \label{eqn:projThreeFieldP}
\end{align}
\end{subequations}
for all $z_{h} \in Z_{h}$ and $q_{h} \in Q_{h}$. Due to the saddle point structure of~\eqref{eqn:projThreeField}, the system is uniquely solvable if the discrete inf-sup condition
\begin{equation*}
\adjustlimits\inf_{q_h \in Q_{h}}\sup_{z_h\in Z_{h}}
\frac{\dd_i(z_h,q_h)}{\|z_h\|_{\cZ} \|q_h\|_{\Q}} \geq \gamma > 0
\end{equation*}
is fulfilled for each~$i=1,\dots,\m$ and $(\,\cdot\,, \cdot\,)_{\cHZ}$ is elliptic on the kernel of~$\dd_i$, see e.g. \cite[Ch.~4.2]{BofBF13}. 
In the subsequent error analysis, we will assume that~\eqref{eqn:projThreeField} attains a unique solution and that the projections satisfy the following approximation properties. 
\begin{assumption}[Spatial projection, network case]
\label{ass:projErrorNetwork}
Consider $u \in \V$, $y_i \in \cZ$, $p_i \in \Q$ and assume that~\eqref{eqn:projThreeField} is well-posed. We assume that the projection errors satisfy 
\begin{subequations}
\label{eqn:projThreeFieldError}
\begin{align}
\|u - \Ru u\|_\V &\lesssim h\,\|\nabla^2 u\|_\cHV, \\
\|y_i - \Ri y_i\|_\cHZ &\lesssim h\,\|\nabla y_i\|_\cHZ, \\ 
\|p_i - \Ri p_i\|_\Q &\lesssim h\,(\|\nabla y_i\|_\cHZ + \|\nabla p_i\|_\Q), 
\end{align}
\end{subequations}
if the derivatives~$\nabla^2 u$, $\nabla y_i$, and~$\nabla p_i$ are bounded in~$\cHV$, $\cHZ$, and~$\Q$, respectively. 
\end{assumption}
\begin{example}
For the spaces~$\V = [H^1_0(\Omega)]^d$, $\cHV = [L^2(\Omega)]^d$, $\cZ = H_0(\ddiv, \Omega)$, $\cHZ = [L^2(\Omega)]^d$, and $\Q = L^2(\Omega)$, \Cref{ass:projErrorNetwork} is satisfied if~$V_h$ equals the $P_1$ Lagrange finite element space, $Z_{h}$ the Raviart-Thomas space $RT_0$, and~$Q_{h}$ the piecewise constant~$P_0$ space. 
For the proofs we refer to~\cite[Ch.~II.7]{Bra07} and~\cite[Th.~3.3]{Dur08}.
\end{example}
%
%
\subsection{Full discretization of the delay system}
As for the two-field model in~\Cref{sec:semiTwoField}, we now analyze the implicit time discretization of the delay PDAE~\eqref{eqn:delay:network}, since this is equal to the proposed semi-explicit scheme~\eqref{eqn:discrete:network}. 
\begin{proposition}
\label{prop:discErrorNetwork}
Suppose Assumptions~\ref{ass:weakCouplingNetwork}, \ref{ass:beta}, and~\ref{ass:projErrorNetwork} and the assumptions of~\Cref{prop:delayNetwork} hold, as well as~$\nabla^2 \bar u \in L^\infty(\cHV)$, $\nabla {\bar y}_i \in W^{1,\infty}(\cHZ)$, and $\nabla {\bar p}_i \in W^{1,\infty}(\Q)$. 
Then, taking initial data $u_h^0\in V_h$ and $p_{i,h}^0\in Q_h$ with 
\[
  \|\Ru u^0 - u_h^0\|_\V + \sum_{i=1}^\m\|\Ri p_i^0 - p^0_{i,h}\|_\Q\, 
  \lesssim\, h
\]
implies that for all $n \le T/\tau$ the solution of the fully discretized system~\eqref{eqn:discrete:network} satisfies 
\[
  \|\bar u(t_n) - u_h^n\|^2_\V 
  + \sum_{i=1}^\m \|\bar p_i(t_n) - p_{i,h}^n\|^2_\Q 
  + \sum_{k=1}^n\, \sum_{i=1}^\m \tau\, \|\bar y_i(t_k) - y_{i,h}^k\|^2_\cHZ\! 
  \lesssim e^{2 t_n} (1+t_{n})\, ( h^2 + \tau^2 ).
\]
\end{proposition}
\begin{proof}
We follow the same approach as in the proof of~\Cref{prop:discErrorTwoField}. 
First, we introduce 
\begin{displaymath}
	\eta_u^n \vcentcolon= \Ru \bar u^n - u^n_h \in V_h,\qquad
	\eta_{y_i}^n \vcentcolon= \Ri \bar y_i^n - y^n_{i,h} \in Z_h,\qquad 
	\eta_{p_i}^n \vcentcolon= \Ri \bar p_i^n - p^n_{i,h} \in Q_h,
\end{displaymath}
where $\bar u^{n}= \bar u(t_{n})$, $\bar y_i^{n}=\bar y_i(t_{n})$, and $\bar p_i^{n}= \bar p_i(t_{n})$ are the solutions of \eqref{eqn:delay:network} and~$\Ru$, $\Ri$ denote the projections defined in~\eqref{eqn:projThreeFieldU} and \eqref{eqn:projThreeField}, respectively. With~\eqref{eqn:discrete:network:a} and \eqref{eqn:delay:network:a}, we compute
\begin{align*}
	a(\eta^{n+1}_u,v_h) - \sum_{i=1}^\m d_i(v_h,\eta^{n+1}_{p_i}) 
	&= \sum_{i=1}^\m \Big[ d_i(v_h,\bar p_i^{n}- \Ri \bar p_i^{n}) 
	- d_i(v_h,\eta^{n+1}_{p_i}- \eta^{n}_{p_i}) \Big]
\end{align*}
for all $v_h\in V_h$. Similarly, we obtain from \eqref{eqn:discrete:network:b} and \eqref{eqn:delay:network:b} that
\begin{equation*}
	(\eta_{y_i}^{n+1},z_h)_\cHZ - \dd_i(z_h,\eta_{p_i}^{n+1})
	= 0
\end{equation*}
for all $z_h \in Z_h$ and $i = 1,\dots,\m$. 
Equations~\eqref{eqn:discrete:network:c} and \eqref{eqn:delay:network:c} yield 
\begin{align*}
	\tau\, \dd_i(\eta^{n+1}_{y_i},q_h) 
	&= \tau\, \dd_i(\bar y_i^{n+1},q_h) - \tau \,\dd_i(y^{n+1}_{i,h},q_h)\\
	&= -d_i(\tau \dot{\bar u}^{n+1},q_h) - c(\tau \dot{\bar p}_i^{n+1},q_h) + d_i(\tau D_\tau u^{n+1}_h,q_h) + c(\tau D_\tau p^{n+1}_{i,h},q_h)\\
	&\phantom{=}\quad + \tau\, \sum_{j\neq i} \beta_{ij} \big((\bar p_i^{n+1}-\Ri \bar p_i^{n+1} + \eta_{p_i}^{n+1}) - (\bar p_j^{n+1}-\Rj \bar p_j^{n+1} + \eta_{p_j}^{n+1}), q_h \big)_\Q 
\end{align*}
for all $q_h \in Q_h$. 
Defining 
\begin{displaymath}
	\theta^{n+1}_u \vcentcolon= \Ru \bar u^{n+1} - \Ru \bar u^n - \tau \dot{\bar u}^{n+1} \qquad\text{and}\qquad
	\theta^{n+1}_{p_i} \vcentcolon= \Ri \bar p_i^{n+1} -\Ri \bar p_i^n -\tau \dot{\bar p}_i^{n+1},
\end{displaymath}
we get with the particular test functions $v_h = \eta_u^{n+1} - \eta_u^n$, $z_h = \eta^{n+1}_{y_i}$, and $q_h = \eta^{n+1}_{p_i}$, 
\begin{multline*}
	a(\eta^{n+1}_u,\eta^{n+1}_u - \eta^n_u) 
	+ \sum_{i=1}^\m c(\eta^{n+1}_{p_i}-\eta^n_{p_i},\eta^{n+1}_{p_i}) 
	+ \tau\, \sum_{i=1}^\m \|\eta^{n+1}_{y_i}\|^2_\cHZ \\ 
	\quad
	\begin{aligned}
	&= \sum_{i=1}^\m \Big[ d_i(\eta^{n+1}_u - \eta^n_u,\bar p^{n}_i- \Ri \bar p_i^{n} - \tau D_\tau\eta^{n+1}_{p_i}) + d_i(\theta^{n+1}_u,\eta^{n+1}_{p_i}) + c(\theta^{n+1}_{p_i},\eta^{n+1}_{p_i})\Big]\\
	&\phantom{=}\quad+ \tau\,\sum_{i=1}^\m\sum_{j\neq i} \beta_{ij}\Big[(\bar p_i^{n+1}-\Ri \bar p_i^{n+1} + \eta_{p_i}^{n+1},\eta_{p_i}^{n+1})_\Q - (\bar p_j^{n+1}-\Rj \bar p_j^{n+1} + \eta_{p_j}^{n+1},\eta_{p_i}^{n+1})_\Q\Big].
	\end{aligned}
\end{multline*}
Using \Cref{lem:symBilinear} for the bilinear forms $a$ and $c$, we obtain 
\begin{multline*}
	\|\eta^{n+1}_u\|^2_a - \|\eta^n_u\|^2_a + \|\tau D_\tau\eta^{n+1}_u\|^2_a + \sum_{i=1}^\m\Big[\|\eta^{n+1}_{p_i}\|^2_c - \|\eta^n_{p_i}\|^2_c + \|\tau D\eta^{n+1}_{p_i}\|^2_c + 2\tau\, \|\eta^{n+1}_{y_i}\|^2_\cHZ\Big] \\
\begin{aligned}
	&= 2\, \sum_{i=1}^\m \Big[ d_i(\tau D_\tau\eta^{n+1}_u,\bar p_i^{n}- \Ri \bar p_i^{n} - \tau D_\tau\eta^{n+1}_{p_i}) + d_i(\theta^{n+1}_u,\eta^{n+1}_{p_i}) + c(\theta^{n+1}_{p_i},\eta^{n+1}_{p_i}) \Big]\\
	&\phantom{=}\quad+2\tau\, \sum_{i=1}^\m\sum_{j\neq i} \beta_{ij}\Big[(\bar p_i^{n+1}-\Ri \bar p_i^{n+1} + \eta_{p_i}^{n+1},\eta_{p_i}^{n+1})_\Q - (\bar p_j^{n+1}-\Rj \bar p_j^{n+1} + \eta_{p_j}^{n+1},\eta_{p_i}^{n+1})_\Q\Big].
\end{aligned}
\end{multline*}
Rewriting the first term on the right-hand side as for the two-field model and using the Taylor expansion for~$\bar p_i^{n}$, we can estimate
\begin{multline*}
	2\, d_i(\tau D_\tau \eta^{n+1}_u,\bar p_i^{n}- \Ri\bar p_i^{n}-\tau D_\tau\eta^{n+1}_{p_i}) \\
	\begin{aligned}
	&\leq \tfrac{C_{d_i}^2}{c_a\,c_c}\|\tau D_\tau \eta^{n+1}_u\|_a^2 + \|\tau D_\tau\eta^{n+1}_{p_i}\|_c^2 
	+ \tfrac{\tau}{\m} \|\eta_u^n\|_a^2 + \tau\tfrac{\m\, C^2_{d_i}}{c_a}\|\dot{\bar p}_i - \Ri \dot{\bar p}_i\|_{L^\infty(\Q)}^2 \\
	&\phantom{\leq}\quad+ 2\, d_i(\eta_u^{n+1},\bar p_i^{n}-\Ri \bar p_i^{n}) - 2\, d_i(\eta_u^{n},\bar p_i^{n-1}-\Ri \bar p_i^{n-1}).
\end{aligned}
\end{multline*}
Further, we have the two estimates 
\begin{align*}
d_i(\theta^{n+1}_u,\eta^{n+1}_{p_i}) 
\leq \tfrac{C_{d_i}^2}{c_c}\tfrac{1}{\tau}\|\theta^{n+1}_u\|_{\V}^2 + \tfrac{\tau}{4} \|\eta^{n+1}_{p_i}\|^2_{c}, \quad
c(\theta^{n+1}_{p_i},\eta^{n+1}_{p_i}) 
\leq \tfrac{C_c^2}{\tau} \|\theta^{n+1}_{p_i}\|^2_\Q + \tfrac{\tau}{4} \|\eta^{n+1}_{p_i}\|^2_c,
\end{align*}
and the double sum including the exchange rates~$\beta_{ij}$ is bounded from above by 
\begin{align*}
  2\tau\beta\, (\m-1)\, 
  \sum_{i=1}^\m \Big[ \|\bar p_i^{n+1}-\Ri \bar p_i^{n+1}\|_\Q^2
  + \tfrac{3}{c_c}\|\eta_{p_i}^{n+1}\|_c^2 \Big].
\end{align*}
We combine the previous estimates and absorb the terms~$\|\tau D_\tau\eta^{n+1}_{p_i}\|_c^2$ and~$\|\tau D_\tau\eta^{n+1}_u\|^2_a$ using \Cref{ass:weakCouplingNetwork} for the latter. 
Further, we apply~$6\beta (\m-1) \le c_c$ from~\Cref{ass:beta}. 
This yields
\begin{multline*}
	\|\eta^{n+1}_u\|^2_a - (1+\tau) \|\eta^n_u\|^2_a 
	+ \sum_{i=1}^\m\Big[ (1-2\tau) \|\eta^{n+1}_{p_i}\|^2_c - \|\eta^n_{p_i}\|^2_c 
	+ 2\tau\, \| \eta^{n+1}_{y_i}\|^2_\cHZ\Big] \\
	- 2\, \sum_{i=1}^\m d_i(\eta_u^{n+1},\bar p_i^{n}-\Ri \bar p_i^{n}) 
	+ 2\, \sum_{i=1}^\m d_i(\eta_u^{n},\bar p_i^{n-1}-\Ri \bar p_i^{n-1}) \\
	\begin{aligned}
	&\lesssim \sum_{i=1}^m\Big[ \tau\,\m\, \|\dot{\bar p}_i-\Rp \dot{\bar p}_i\|_{L^\infty(\Q)}^2 + 2\tau\beta (\m-1)\|\bar p_i^{n+1}-\Ri \bar p_i^{n+1}\|_\Q^2 + \tfrac{1}{\tau} \|\theta^{n+1}_u\|^2_\V + \tfrac{1}{\tau} \|\theta^{n+1}_{p_i}\|^2_\Q \Big].
	\end{aligned}
\end{multline*}
As in the proof of~\Cref{prop:discErrorTwoField} we can apply~\Cref{ass:projErrorNetwork} to bound~$\theta^{n+1}_u$ and~$\theta^{n+1}_{p_i}$, leading to 
\begin{align*}
\|\theta^{n+1}_u\|_{\V} 
&\lesssim \tau h\, \|\nabla^2\dot{\bar u}\|_{L^\infty(\cHV)} + \tau^2\|\ddot{\bar u}\|_{L^\infty(\V)}, \\
\|\theta^{n+1}_{p_i}\|_\Q 
&\lesssim \tau h\, \|\nabla\dot{\bar y}_i\|_{L^\infty(\cHZ)} 
+ \tau h\, \|\nabla\dot{\bar p}_i\|_{L^\infty(\Q)} 
+ \tau^2 \|\ddot{\bar p}_i\|_{L^\infty(\Q)}.
\end{align*}
Finally, the discrete version of the Grönwall lemma gives 
\[
\|\eta^{n}_u\|^2_\V 
+ \sum_{i=1}^\m \|\eta^{n}_{p_i}\|^2_\Q 
+ \sum_{k=1}^n\, \sum_{i=1}^\m \tau\, \|\eta^{k}_{y_i}\|^2_\cHZ \ 
\lesssim e^{2 t_n} (1+t_{n})\, ( h^2 + \tau^2 )
\]
such that the assertion follows by~\Cref{ass:projErrorNetwork} and the triangle inequality.
\end{proof}
%
%
\subsection{Convergence of the semi-explicit scheme}
The combination of~\Cref{prop:delayNetwork,prop:discErrorNetwork} provides the desired convergence property of the semi-explicit scheme~\eqref{eqn:discrete:network}.   
\begin{theorem}[Convergence, network model]
\label{thm:networkConvergence}
Suppose Assumptions~\ref{ass:weakCouplingNetwork}, \ref{ass:beta}, and~\ref{ass:projErrorNetwork} hold. 
Further, let the right-hand sides~$f\colon [0,T] \to \cHV$ and~$g_i\colon [0,T] \to \Q$ be sufficiently smooth.  
Then, with~$(u, y_i, p_i)$ being the solution of the original system~\eqref{eqn:network} and $u^n_h\in V_h$, $y^n_{i,h}\in Z_h$, $p^n_{i,h}\in Q_h$ the fully discrete approximations obtained by~\eqref{eqn:discrete:network} for $n \le T/\tau$ and initial data~$u_h^0\in V_h$, $p_{i,h}^0\in Q_h$ with
\[
\|\Ru u^0 - u_h^0\|_\V + \sum_{i=1}^\m\|\Ri p_i^0 - p^0_{i,h}\|_\Q\, 
\lesssim\, h
\]
we obtain the error estimate 
\[
\| u(t_n) - u_h^n\|^2_\V 
+ \sum_{i=1}^\m \| p_i(t_n) - p_{i,h}^n\|^2_\Q 
+ \sum_{k=1}^n\, \sum_{i=1}^\m \tau\, \| y_i(t_k) - y_{i,h}^k\|^2_\cHZ 
\lesssim e^{C t_n} (1+t_{n})\, ( h^2 + \tau^2 )
\]
with a constant~$C$ depending on~$\beta$, but independent of~$\tau$ and~$h$. 
\end{theorem}
\begin{proof}
We define the history functions~$\Phi_i$ as in~\eqref{eqn:historyNetwork}. As in the two-field case, the assumptions on the data and the assumed~$H^2$-regularity imply that the solution of the related delay system~\eqref{eqn:delay:network} stays bounded. 
Thus, following the procedure presented in~\Cref{sec:neutralDelayPDE}, we conclude~$\nabla^2 \bar u \in L^\infty(\cHV)$, $y_i, \bar y_i \in W^{2,\infty}(\cHZ)$, $\nabla {\bar y}_i \in W^{1,\infty}(\cHZ)$, ${\bar p}_i\in W^{2,\infty}(\Q)$, and~$\nabla {\bar p}_i \in W^{1,\infty}(\Q)$. 
With this, all assumptions of~\Cref{prop:delayNetwork,prop:discErrorNetwork} are satisfied and the stated estimate follows from the previous results and the triangle inequality. 
Note that the estimate of the~$y_i$-terms requires an application of the trapezoidal rule for the function~$\sum_{i=1}^\m \Vert \bar y_i(t) - y_i(t)\Vert^2_{\cHZ}$. 
%
\end{proof}
%
\section{Numerical Examples}\label{sec:num}
This section is devoted to the numerical illustration of the convergence results presented in~\Cref{thm:twoFieldConvergence,thm:networkConvergence} and corresponding runtime comparisons. 
Furthermore, we show that the weak coupling condition is sharp and actually a necessary condition for the convergence of the semi-explicit scheme. 

All computations use a FEniCS finite element implementation and have been performed on an HPC Infiniband cluster. 
%
%
\subsection{Linear poroelasticity}\label{sec:num:poro} 
We test the semi-explicit time-integration with the linear poroelasticity example presented in \Cref{exp:poro}. The parameters for the simulation (ommitting the units) are given by
\begin{center}
\begin{tabular}{ccccc}
	$\lambda$ & $\mu$ & $\tfrac{\kappa}{\nu}$ & $\tfrac{1}{M}$ & $\alpha$ \\\toprule
	\num{1.2e10} & \num{6.0e9} & \num{6.33e2} & \num{7.8e3} & \num{0.79}
\end{tabular}.
\end{center}
The simulation is performed in the two-dimensional unit square~$\Omega=(0,1)^2$ with final time $T = 10$. The source terms and the initial condition are chosen as 
\begin{align*}
	f \equiv 0, \qquad 
	g(t) = 10\, \mathrm{e}^t, \qquad \text{and}\qquad 
	p^0(x,y) = 3000\, x(1-x)\, y(1-y).
\end{align*}
For the error analysis, we compute a reference solution with mesh size $h = \num{1.95e-3}$ (with standard $P_1$ finite elements and homogeneous Dirichlet boundary conditions, cf.~\Cref{ex:P1Two}) and time step size~$\tau = \num{4.88e-3}$. The relative errors in the energy norms~$\|\cdot\|_a$ and~$\|\cdot\|_c$ at the final time are depicted in \Cref{fig:poroConvergence}.

\begin{figure}[ht]
	\centering
	\begin{subfigure}[t]{.45\linewidth}
\begin{tikzpicture}

\begin{axis}[
width=2.8in,
height=2.5in,
log basis x={10},
log basis y={10},
tick align=outside,
tick pos=left,
x grid style={white!69.01960784313725!black},
xmin=0.00821187905521205, xmax=0.37162722343835,
xmode=log,
xtick style={color=black},
xlabel = {step size $\tau$},
ylabel = {relative error},
x label style={at={(axis description cs:0.5,-0.05)},anchor=north},
y label style={at={(axis description cs:-0.00,.5)},anchor=south},
y grid style={white!69.01960784313725!black},
ymin=0.000694047534352092, ymax=0.41802792863522,
ymode=log,
ytick style={color=black}
]
\addplot [semithick, color0, mark=asterisk, mark size=3, mark options={solid}]
table {%
0.3125 0.121311205082746
0.15625 0.102696449545496
0.078125 0.0970544552733723
0.0390625 0.0955354687310521
0.01953125 0.0951453139492023
0.009765625 0.0950469941313551
};
\addplot [semithick, color0]
table {%
0.3125 0.121311205082729
0.15625 0.102696449545496
0.078125 0.097054455273375
0.0390625 0.0955354687310525
0.01953125 0.0951453139492014
0.009765625 0.0950469941313546
};
\addplot [semithick, color1, dashed, mark=asterisk, mark size=3, mark options={solid}]
table {%
0.3125 0.0887035841144379
0.15625 0.0610737863269325
0.078125 0.0512738213666369
0.0390625 0.0485000503960298
0.01953125 0.0478172064111215
0.009765625 0.0476684848808265
};
\addplot [semithick, color1, dashed]
table {%
0.3125 0.0887035841145695
0.15625 0.0610737863269237
0.078125 0.0512738213666394
0.0390625 0.0485000503960194
0.01953125 0.0478172064111184
0.009765625 0.0476684848808288
};
\addplot [semithick, color2, mark=asterisk, mark size=3, mark options={solid}]
table {%
0.3125 0.0782396749695912
0.15625 0.0446228089718713
0.078125 0.02994689361992
0.0390625 0.0249817313415543
0.01953125 0.0236751392168991
0.009765625 0.0233977797877669
};
\addplot [semithick, color2]
table {%
0.3125 0.0782396749692228
0.15625 0.0446228089716778
0.078125 0.0299468936199403
0.0390625 0.0249817313413539
0.01953125 0.0236751392168324
0.009765625 0.0233977797877544
};
\addplot [semithick, color3, dashed, mark=asterisk, mark size=3, mark options={solid}]
table {%
0.3125 0.0753644122799604
0.15625 0.0393899273103068
0.078125 0.0214208044770823
0.0390625 0.0136795845731956
0.01953125 0.0111403627727026
0.009765625 0.0105517342229072
};
\addplot [semithick, color3, dashed]
table {%
0.3125 0.0753644122807091
0.15625 0.0393899273074157
0.078125 0.0214208044757615
0.0390625 0.013679584573443
0.01953125 0.0111403627731948
0.009765625 0.0105517342226706
};
\addplot [thick, black, dotted]
table {%
0.3125 0.0625
0.009765625 0.001953125
};
\end{axis}

\end{tikzpicture}
	\end{subfigure}\qquad
	\begin{subfigure}[t]{.45\linewidth}
\begin{tikzpicture}

\begin{axis}[
width=2.8in,
height=2.5in,
legend cell align={left},
legend style={at={(0.03,0.97)}, anchor=north west, draw=white!80.0!black},
log basis x={10},
log basis y={10},
tick align=outside,
tick pos=left,
x grid style={white!69.01960784313725!black},
xmin=0.00821187905521205, xmax=0.37162722343835,
xmode=log,
xtick style={color=black},
xlabel = {step size $\tau$},
x label style={at={(axis description cs:0.5,-0.05)},anchor=north},
y grid style={white!69.01960784313725!black},
ymin=0.000694047534352092, ymax=0.41802792863522,
ymode=log,
ytick style={color=black},
legend style={
	at={(0.5,-0.1)},
	anchor=north,
	/tikz/column 2/.style={
                column sep=10pt,
            },
    /tikz/column 4/.style={
                column sep=10pt,
            },
},
legend columns=3,
legend to name=legPoroConvergence, 
]
\addplot [semithick, color0, mark=asterisk, mark size=3, mark options={solid}]
table {%
0.3125 0.0638420958581949
0.15625 0.0318247749045041
0.078125 0.0150274970880722
0.0390625 0.00659476605964158
0.01953125 0.00290171570855177
0.009765625 0.0023091944455552
};
\addlegendentry{$h=\num{3.12e-2}$}
\addplot [semithick, color0,forget plot]
table {%
0.3125 0.0638420958581754
0.15625 0.0318247749045101
0.078125 0.0150274970880876
0.0390625 0.00659476605964897
0.01953125 0.00290171570854567
0.009765625 0.00230919444555641
};
\addplot [semithick, color1, dashed, mark=asterisk, mark size=3, mark options={solid}]
table {%
0.3125 0.0650923472163485
0.15625 0.032961601891414
0.078125 0.0160392525533389
0.0390625 0.00736748588774769
0.01953125 0.0030023350135265
0.009765625 0.000928420010110678
};
\addlegendentry{$h=\num{1.56e-02}$}
\addplot [semithick, color1, dashed, forget plot]
table {%
0.3125 0.0650923472164636
0.15625 0.0329616018913813
0.078125 0.0160392525533255
0.0390625 0.0073674858876821
0.01953125 0.00300233501347866
0.009765625 0.000928420010146472
};
\addplot [semithick, color2, mark=asterisk, mark size=3, mark options={solid}]
table {%
0.3125 0.0654123659630795
0.15625 0.0332605881737111
0.078125 0.0163237573572178
0.0390625 0.00763597813169851
0.01953125 0.00323798136066624
0.009765625 0.00103044979962129
};
\addlegendentry{$h=\num{7.81e-3}$}
\addplot [semithick, color2, forget plot]
table {%
0.3125 0.065412365962694
0.15625 0.0332605881734673
0.078125 0.0163237573572018
0.0390625 0.00763597813115645
0.01953125 0.0032379813602623
0.009765625 0.00103044979939312
};
\addplot [semithick, color3, dashed, mark=asterisk, mark size=3, mark options={solid}]
table {%
0.3125 0.065492830443623
0.15625 0.0333362267215712
0.078125 0.016396706284635
0.0390625 0.00770708475216959
0.01953125 0.00330688607442831
0.009765625 0.0010930939435466
};
\addlegendentry{$h=\num{3.91e-3}$}
\addplot [semithick, color3, dashed, forget plot]
table {%
0.3125 0.0654928304444942
0.15625 0.0333362267191371
0.078125 0.0163967062835035
0.0390625 0.00770708475270424
0.01953125 0.00330688607589657
0.009765625 0.00109309394202659
};
\addplot [thick, black, dotted]
table {%
0.3125 0.1250
0.009765625 0.00390625
};
\addlegendentry{linear}
\end{axis}

\end{tikzpicture}
	\end{subfigure}\\
	\ref{legPoroConvergence} 
	\caption{Relative error at final time~$T=10$ for the implicit (stars) and the semi-explicit Euler scheme (solid/dashed line) for different mesh sizes~$h$. Left: displacement $u$. Right: pressure $p$.}
	\label{fig:poroConvergence}
\end{figure}
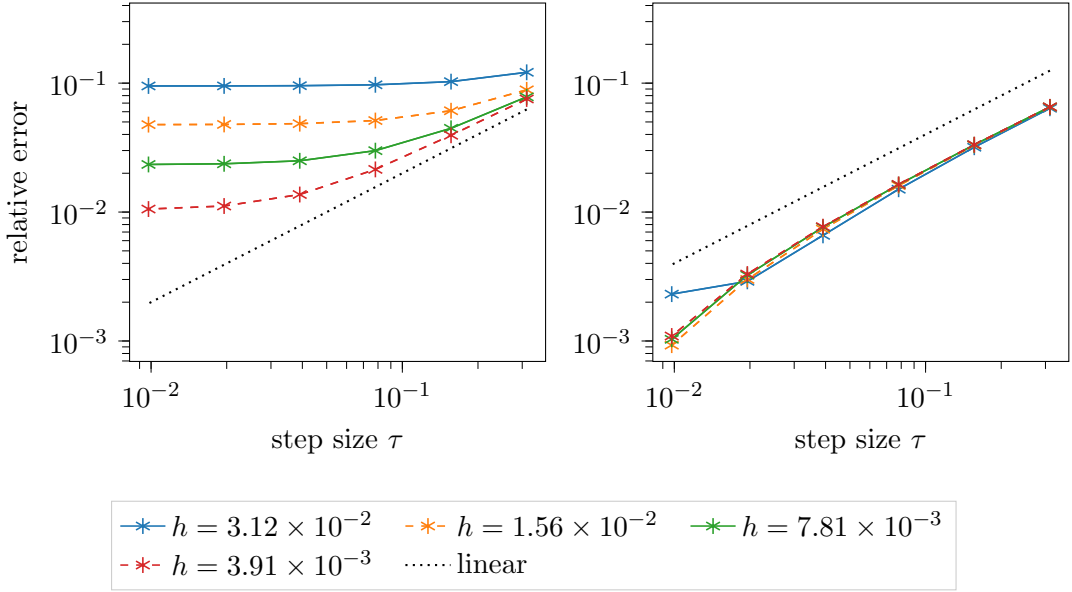

For the pressure variable (right plot in \Cref{fig:poroConvergence}), we observe a linear decay of the error with respect to the time step size in agreement with \Cref{thm:twoFieldConvergence}. The relative error in the displacement (left plot in \Cref{fig:poroConvergence}) is dominated by the spatial discretization error, which also shows the predicted linear decay.  
Thus, the numerical experiment confirms our theoretical findings. It is worth to mention that the semi-explicit Euler performs very similarly as the implicit Euler with a negligible difference.
%
%
\subsection{A network example}
We consider a simple network example with $\m = 4$, i.e., we have four pressure variables. 
The used parameters are mainly motivated by the ones considered in \cite{VarCTHLTV16,JCLT19} and given by 
\begin{center}
\begin{tabular}{cccccccc}
	$\lambda$ & $\mu$ & $\tfrac{\kappa_1}{\nu_1} = \tfrac{\kappa_2}{\nu_2}$ & $\tfrac{\kappa_3}{\nu_3}$ & $\tfrac{\kappa_4}{\nu_5}$ & $\tfrac{1}{M}$ & $\alpha_i$ \\\toprule
	\num{7786.42} & \num{3337.037} & \num{3.75e-4} & \num{1.57e-5} & \num{3.75e-5} & \num{4.50e-2} & \num{0.99}\\
\end{tabular}.
\end{center}
Further, the parameters $\beta_{ij},\, i,j \in \{1,2,3,4\}$ are set to zero, besides 
\begin{equation*}
 \beta_{12} = \beta_{24} = 1.5 \times 10^{-19}, \qquad 
 \beta_{23} = 2 \times 10^{-19}, \qquad
 \beta_{34} = 1 \times 10^{-13}.
\end{equation*}
The simulation is performed in the domain ${\Omega = (0,1)^2 \setminus B_{0.25}((0.5,0.5))}$ with~$T = 10$ and using the $P_1$ finite element space for the displacement, the Raviart-Thomas space of lowest order~$RT_0$ for the fluid fluxes, and the space $P_0$ of piecewise constants for the pressures. The source terms are given by $f \equiv 0$ and $g \equiv 0$ and the initial pressures are chosen as $p^0_2 \equiv p^0_4 \equiv 650$, $p_3^0 \equiv 1000$ and
\begin{align*}
p_1^0(x,y) = \mathbb{1}_{\Omega_p}\big(13300-3238400\,((x-0.75)^2+(y-0.75)^2)\big)
+ \mathbb{1}_{\Omega\setminus\Omega_p}650
\end{align*}
with $\Omega_p = {\{(x,y) \in \Omega\colon (x-0.75)^2+(y-0.75)^2 \leq 1/256\}}$. This means that three initial pressures are constant and $p_1^0$ has a local peak. 

As predicted by~\Cref{thm:networkConvergence}, we have linear convergence in time and space, very similar to the previous example. 
The respective runtimes for the implicit and semi-explicit schemes are given in~\Cref{tab:runtimeNetwork}. The numbers show that a significant percentage of the computation time can be saved when computing with the semi-explicit scheme. 
This is of particular value for small mesh and step sizes $h$ and $\tau$. 
Further note that we did not yet exploit the fact that the decoupled nature of the semi-explicit method facilitates the use of preconditioned iterative methods. 
\newcolumntype{P}{R{1.3cm}}
\begin{table}[ht]
	\caption{Runtime comparison (in seconds)}\label{tab:runtimeNetwork}
	\begin{tabular}{rPPPPP}
		$h = \tau$ & $2^{-4}$ & $2^{-5}$ & $2^{-6}$ & $2^{-7}$ & $2^{-8}$
		\\\toprule
		implicit & $13.2$ & $65.8$ & $471.0$& $4027.2$ & $29921.1$
		\\
		semi-explicit & $9.7$ & $51.6$ & $314.6$& $2544.2$ & $19994.4$
		\\\bottomrule
		\textbf{reduction (in \%)} & $26.5$ & $21.6$ & $33.2$ & $36.8$ & $33.2$
	\end{tabular}
\end{table}
\begin{figure}[htp]
	\centering
	\begin{subfigure}[c]{0.49\textwidth}
		\includegraphics[height=5.8cm]{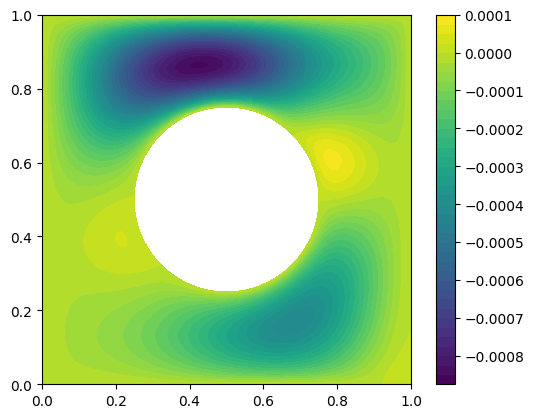}
		\subcaption{displacement $u$ (first component)}
	\end{subfigure}
	\begin{subfigure}[c]{0.49\textwidth}
		\includegraphics[height=5.8cm]{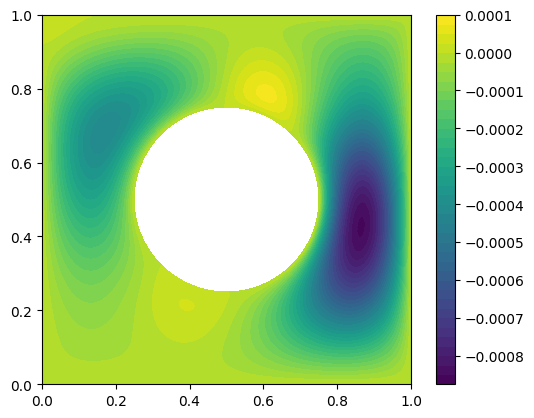}
		\subcaption{displacement $u$ (second component)}
	\end{subfigure}
	\begin{subfigure}[c]{0.49\textwidth}
		\vspace{0.1cm}
		\includegraphics[height=5.8cm]{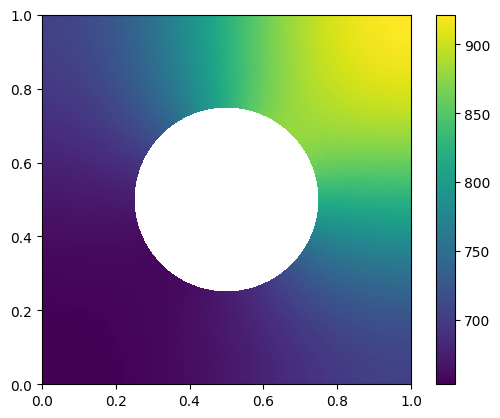}
		\subcaption{pressure $p_1$}
	\end{subfigure}
	\begin{subfigure}[c]{0.49\textwidth}
		\includegraphics[height=5.94cm]{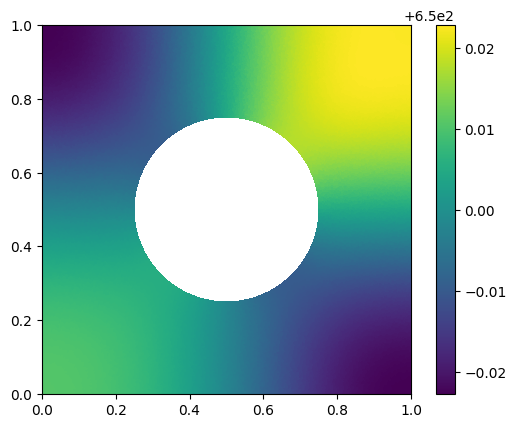}
		\subcaption{pressure $p_2$}
	\end{subfigure}
	\begin{subfigure}[c]{0.49\textwidth}
		\includegraphics[height=5.8cm]{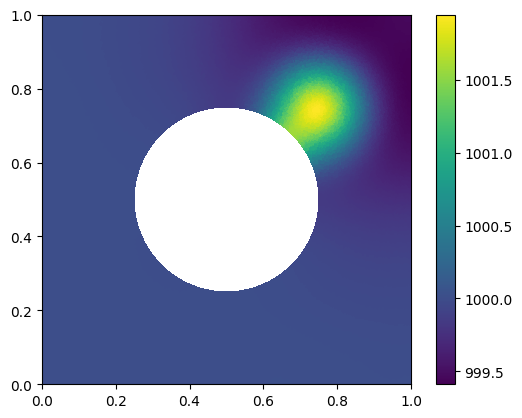}
		\subcaption{pressure $p_3$}
	\end{subfigure}
	\begin{subfigure}[c]{0.49\textwidth}
		\includegraphics[height=5.8cm]{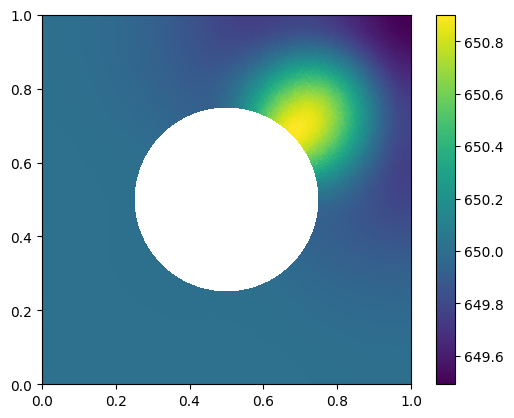}
		\subcaption{pressure $p_4$}
	\end{subfigure}
	\caption{Displacement $u$ and pressures $p_i$, $i \in \{1,2,3,4\}$ at time ${T = 10}$ for the network case, computed with $h = \tau = 2^{-7}$.} 
	\label{fig:plotsNetwork}
\end{figure}

The displacement and the pressures at final time $T = 10$ for the choice ${h=\tau = 2^{-7}}$ are shown in~\Cref{fig:plotsNetwork}.  
One can observe that the high pressure peak in the pressure variable $p_1$ in the initial condition starts to average out across the whole domain and also has an influence on the other pressure variables, especially on~$p_3$ and~$p_4$ where small local pressure increases can be observed. The effect on the pressure $p_2$ is only minor but not as local as for the pressures $p_3$ and $p_4$. Moreover, the changes in the pressures also lead to a deformation of the object originating from the location of the original pressure peak.
%
%
\subsection{Sharpness of the weak coupling condition}
\label{sec:num:weakCoupling}
We conclude our numerical examples with an investigation of the weak coupling condition in \Cref{ass:weakCouplingTwo}. To this end, we consider a toy problem of the form \eqref{eqn:two} with $\V = \cHV = \R^3$, $\Q = \cHQ = \R^1$ and bilinear forms
\begin{align*}
	a(u,v) = v^TAu,\qquad d(v,p) = \omega\, p^TDv, \qquad c(p,q) = q^TCp, \qquad b(p,q) = q^TBp
\end{align*} 
with matrices 
\begin{align*}
	A \vcentcolon= \begin{smallbmatrix}
		\ 2 & -1 & \ 0\\
		-1 &\ 2 & -1\\
		\ 0 & -1 &\ 2
	\end{smallbmatrix},\qquad 
	D \vcentcolon= \big[\, 1\ \ 2\ \ 3\, \big], \qquad 
	C \vcentcolon= 1,\qquad \text{and}\qquad 
	B\vcentcolon= 1.
\end{align*}
The constants $c_a$ and $c_c$ are given by the smallest eigenvalue of $A$ and $C$, respectively, i.e., $c_a \approx 0.586$ and $c_c = 1$. The continuity constant $C_d(\omega)$ is given by the spectral norm of~$\omega\, D$. In particular, we have $C_d(\omega) = |\omega|\|D\|_2 \approx 3.742\,|\omega|$, i.e., the weak coupling condition from \Cref{ass:weakCouplingTwo} requires $\omega \in [-0.2046,0.2046]$. Moreover, in view of \Cref{ex:spectralRadius}, the necessary and sufficient condition for the asymptotic stability of the related delay equation requires $\omega \in (-0.2182,0.2182)$. We test our semi-explicit scheme with different step sizes~$\tau$ and compute the relative error at the final time $T = 1$. For the forcing functions, we choose
\begin{displaymath}
	f^T \equiv \big[\, 1\ \ 1\ \ 1\, \big]\qquad\text{and}\qquad 
	g(t) = \sin(t).
\end{displaymath}
The results are presented in \Cref{fig:weakCoupling}, where the two critical values are represented with dashed lines.
\begin{figure}[ht]
	\centering
	\input{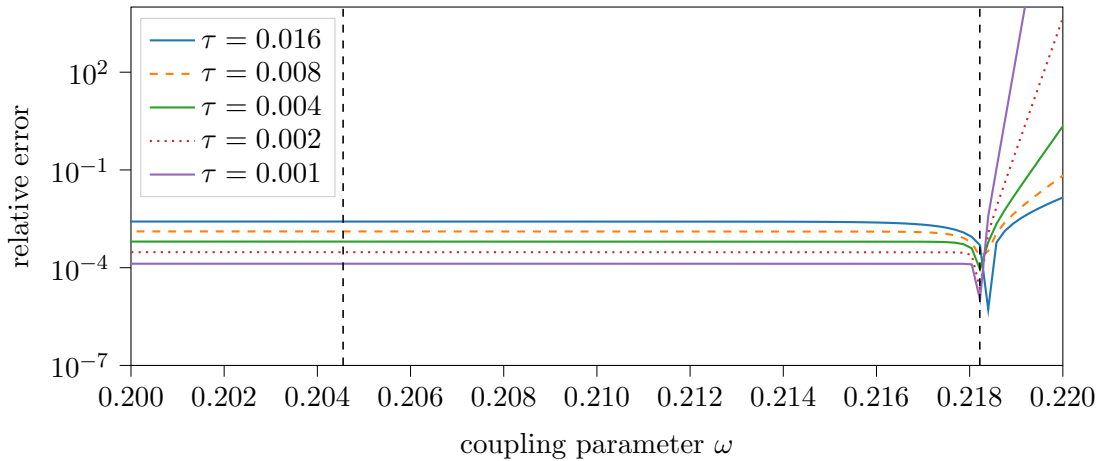}
	\caption{Relative error at the final time point~$T=1$ for different coupling parameters $\omega$ and different time step sizes $\tau$.}
	\label{fig:weakCoupling}
\end{figure}
As expected from \Cref{thm:twoFieldConvergence}, the semi-explicit Euler schemes approximates the true solution well for all $\omega$ that satisfy the weak coupling condition. Independent of the step size $\tau$ we observe that the semi-explicit schemes fail when the related delay equation \eqref{eqn:delay:two} becomes asymptotically unstable.
%
%
\section{Conclusions}
Within this paper, we have proven first-order convergence in time and space of the combination of the semi-explicit Euler scheme with a conforming finite element discretization. This result enables a more efficient time-stepping scheme as the linear system, which needs to be solved in every time step, decouples. 

For the convergence analysis, we have employed a new technique which links the semi-explicit discretization to an implicit discretization of a related delay system. This approach also generates theoretical insight and explains the required weak coupling condition. Convergence of the method is proven for the classical two-field formulation (including poroelasticity) as well as for multiple-network systems which are of high interest in medical applications. 
The theoretical results are illustrated by three numerical experiments. 
%
%
\section*{Acknowledgements} 
R.~Maier gratefully acknowledges support by the German Research Foundation (DFG) in the Priority Program 1748 \emph{Reliable simulation techniques in solid mechanics} (PE2143/2-2). 
The work of B.~Unger is supported by the German Research Foundation (DFG) Collaborative Research Center 910 \emph{Control of self-organizing nonlinear systems: Theoretical methods and concepts of application}, project number 163436311.

Major parts of the paper were evolved at CIRM in Luminy within a {\em Recherches en Binôme} (Research in Pairs) stay in August 2019. We are grateful for the invitation and kind hospitality. 
Further, we thank C.~Carstensen (HU Berlin) for bringing up the idea of a semi-explicit discretization at CMAM-8 in Minsk. 
%
%
\newcommand{\etalchar}[1]{$^{#1}$}

%
%
\appendix 

\section{Stability of the delay equation for the two-field model}
\label{sec:neutralDelayPDE}
In order to prove \Cref{prop:delayTwoField}, we need $\bar p\in W^{2,\infty}(\cHQ)$, where $(\bar{u},\bar{p})$ is the solution of the delay PDAE~\eqref{eqn:delay:two}. 
In the following we assume smooth data and show that this guarantees the existence of a solution and a uniform bounded in~$W^{2,\infty}(\cHQ)$. 
We expect, however, that this result can be proven with weaker assumptions on the data, for instance, by extending the technique of~\cite{BelGZ99} to the PDE setting. Since the main focus of this article is the convergence result for the semi-explicit scheme, we consider this future research.

For the stability analysis we consider a prototypical system of the form 
\begin{subequations}
	\label{eqn:neutralDelayPDE}
	\begin{align}
		\label{eqn:neutralDelayPDE:eq}
  		\langle \dot{\bar{p}}(t) + \calK \bar{p}(t),q \rangle 
  		&= \langle \omega\, \dot{\bar{p}}(t-\tau) + \gt(t), q\rangle, &t\in(0,T],\\
  		\label{eqn:neutralDelayPDE:ic}
  		\bar{p}(t) 
  		&= \Phi(t), &t\in[-\tau,0]
	\end{align}
\end{subequations}
for all test functions~$q\in \Q$ and with smooth history function~$\Phi\colon[-\tau,0]\to \cHQ$, source term~$\gt\colon [0,T] \to \Q^*$, a constant~$\omega\in \R$, and a linear, bounded operator~$\calK\colon \Q\to\Q^*$ that satisfies a G\aa rding inequality, i.e.,
 \begin{equation}
	\label{eqn:Gaarding}
	\langle \calK q,q\rangle \geq \Kalpha \|q\|_\Q^2 - \Kbeta\|q\|_{\cHQ}^2
\end{equation}
for real constants $\Kalpha>0$, $\Kbeta\geq 0$ and all $q\in \Q$. 

Thus, we first need to reduce the delay PDAE~\eqref{eqn:delay:two} to such a parabolic problem with neutral delay. With the operators~$\calA$, $\calB$, $\calC$, $\calD$ corresponding to the bilinear forms $a$, $b$, $c$, $d$, respectively, system~\eqref{eqn:delay:two} can be written as 
\[
  \calA \bar u - \calD^* \bar p(\,\cdot-\tau) = f \ \text{ in }\V^*, \qquad
  \calD \dot{\bar u} + \calC \dot{\bar p} + \calB \bar p = g \ \text{ in }\Q^*.
\]
Since~$\calA\colon \V \to \V^*$ is an invertible operator, we can differentiate the first equation in time and insert it into the second equation. 
Using also the invertibility of~$\calC$, we get 
\[
  \dot{\bar p} + \calC^{-1}\calB \bar p 
  = \tilde g - \calC^{-1}\calD \calA^{-1}\calD^* \dot{\bar p}(\,\cdot-\tau)  
\]
with~$\tilde g := \calC^{-1}g - \calC^{-1} \calD \calA^{-1} \dot f$. 
Considering~$d$ as bilinear form~$d\colon \V \times \cHQ \to \R$, we note that the operator~$\calC^{-1} \calD \calA^{-1}\calD^*\colon \cHQ \to \cHQ$ is bounded with constant~$\omega := C_d^2/(c_a c_c)$. For the stability analysis we can thus replace this operator by the constant~$\omega$.

We like to emphasize that the reduction from the PDAE~\eqref{eqn:delay:two} to the parabolic problem~\eqref{eqn:neutralDelayPDE} comes along with the differentiation of the delay term. Although~\eqref{eqn:delay:two} seems to be of retarded type at first sight (only~$\bar p(\,\cdot-\tau)$ appears), the reformulation shows that it is in fact {\em neutral}. This is due to the fact that the delay term appears in the elliptic equation and thus, in terms of DAEs, in the constraint. 

Following \cite{AltZ18}, we call a function $\bar{p}\in C([0,T];\cHQ\!)\cap L^2(\Q)$ with $\dot{\bar{p}}\in L^2(\Q^*)$ a weak solution of \eqref{eqn:neutralDelayPDE}, if $\bar{p}(0^+) = \bar{p}^0\in \cHQ$ and 
\begin{displaymath}
	\widetilde{p}(t) \vcentcolon= \begin{cases}
		\Phi(t), & t\in[-\tau,0),\\
		\bar{p}, & t\in[0,T]
	\end{cases}
\end{displaymath}
satisfies \eqref{eqn:neutralDelayPDE} in the variational sense.

\begin{proposition}
\label{prop:neutralDelayPDEexistence}
Consider the initial trajectory problem \eqref{eqn:neutralDelayPDE} with $\gt\in L^2(\Q^*)$ and history function~$\Phi\in C^1([-\tau,0];\cHQ\!)$. 
If the bilinear from associated with $\calK$ satisfies a G\aa{}rding inequality~\eqref{eqn:Gaarding}, then \eqref{eqn:neutralDelayPDE} possesses a unique weak solution.
\end{proposition}

\begin{proof}
On the interval $I_1 \vcentcolon= [0,\tau]$ we have to solve the initial value problem
\begin{align*}
	\mosdp{1}(t) + \calK \mosp{1}(t) 
	&= \gt(t) + \omega\, \dot{\Phi}(t-\tau),\\
	\mosp{1}(0) &= \Phi(0)\in \cHQ.
\end{align*}
The assumptions imply $\gt + \omega\,\dot{\Phi}(\cdot-\tau)\in L^2(0,\tau,\Q^*)$ and thus the theorem of Lions-Tartar \cite[Lem.~19.1]{Tar06} guarantees a unique solution $\mosp{1} \in C([0,\tau];\cHQ)\cap L^2(0,\tau;\Q)$ with $\mosdp{1}\in L^2(0,\tau;\Q^*)$. In particular, we conclude $\mosp{1}(\tau) \in \cHQ$. Applying Bellmann's method of steps (cf.~\cite[Ch.\,3.4]{BelZ03}), we inductively infer that the sequence
\begin{align*}
	\mosdp{n} + \calK \mosp{n} 
	&= \mos{\gt}{n} + \omega\, \mosdp{n-1},\\
	\mosp{n}(0) &= \mosp{n-1}(\tau)
\end{align*}
with $\mosp{0}(t) \vcentcolon= \Phi(t-\tau)$ and $\mos{\gt}{n}(t) \vcentcolon= \gt(t+(n-1)\tau)$ for $t\in[0,\tau]$ possess unique solutions $\mosp{n}\in C([0,\tau];\cHQ)\cap L^2(0,\tau;\Q)$ with $\mosdp{n}\in L^2(0,\tau;\Q^*)$. The result follows by defining $\bar{p}(t) \vcentcolon= \mosp{n}(t+(n-1)\tau)$ for $t\in[(n-1)\tau,n\tau]$.
\end{proof}
\begin{coroll}
\label{lem:smoothSolutionNeutralPDE}
Consider the initial trajectory problem \eqref{eqn:neutralDelayPDE} with $\gt\in C^1([0,T];\cHQ\!)$ and suppose that the history function $\Phi\in C^2([-\tau,0];\cHQ\!)$ satisfies~$\calK\Phi(0)\in\cHQ$ and the splicing condition
\begin{equation}
	\label{eqn:splicingCondition}
	\dot{\Phi}(0) + \calK \Phi(0) = \gt(0) + \omega\, \dot{\Phi}(-\tau).		
\end{equation}
Further, let the bilinear form associated with the operator $\calK$ satisfy a G\aa{}rding inequality~\eqref{eqn:Gaarding}. Then the weak solution $\bar{p}$ of \eqref{eqn:neutralDelayPDE} satisfies $\bar{p}\in C^1([0,T];\cHQ\!)\cap H^1(\Q)$ with $\ddot{\bar{p}}\in L^2(\Q^*)$.
\end{coroll}
\begin{proof}
We differentiate equation \eqref{eqn:neutralDelayPDE:eq}, which (formally) yields
\begin{subequations}
\label{eqn:neutralDelayPDE:derivative}
\begin{align}
	\dot{\bar{q}}(t) + \calK \bar{q}(t) 
	&= \dot{\gt}(t) + \omega\, \dot{\bar{q}}(t-\tau), & t&>0,\\
	\bar{q}(t) &= \dot{\Phi}(t), & t&\in[-\tau,0].
\end{align}
\end{subequations}
The assumptions imply $\dot{\gt}\in C([0,T];\cHQ)$ and $\dot \Phi\in C^1([-\tau,0];\cHQ)$ and thus~\Cref{prop:neutralDelayPDEexistence} establishes the existence of a weak solution $\bar{q}$ of \eqref{eqn:neutralDelayPDE:derivative}. From the splicing condition \eqref{eqn:splicingCondition} and the continuity of the weak solution $\bar{p}$ of \eqref{eqn:neutralDelayPDE} we obtain
\begin{align*}
	\dot{\bar{p}}(0^+) 
	&= -\calK \bar{p}(0) + \gt(0) + \omega\, \dot{\Phi}(-\tau) 
	= -\calK \Phi(0) + \gt(0) + \omega\, \dot{\Phi}(-\tau) 
	= \dot{\Phi}(0) 
	= \bar{q}(0),
\end{align*}
which establishes $\dot{\bar{p}} = \bar{q}$.
\end{proof}
\begin{remark}
If the data is sufficiently smooth and the splicing condition \eqref{eqn:splicingCondition} is also satisfied for derivatives of the history function, i.e.,
\begin{displaymath}
	\Phi^{(\ell+1)}(0) + \calK \Phi^{(\ell)}(0) = \gt^{(\ell)}(0) + \omega\, \Phi^{(\ell+1)}(-\tau)		
\end{displaymath}
for $\ell\in\mathbb{N}$, then by repeating the arguments in the proof of \Cref{lem:smoothSolutionNeutralPDE} for derivatives of $\bar{p}$, we obtain a smooth solution of the initial trajectory problem \eqref{eqn:neutralDelayPDE}.
\end{remark}
\begin{proposition}
\label{prop:neutralDelayPDEStability}
Let $(\bar{u},\bar{p})$ denote a smooth solution of the initial trajectory problem~\eqref{eqn:neutralDelayPDE} with sufficiently smooth data. Moreover, assume that $\calK$ satisfies a G\aa{}rding inequality~\eqref{eqn:Gaarding} and that the derivatives of the history function~$\Phi$ and the right-hand side~$\gt$ are uniformly bounded, i.e., there exist constants $C_\Phi$ and $C_{\gt}$ such that
\begin{displaymath}
	\|\Phi^{(j)}\|_{L^\infty(-\tau,0;\cHQ)}^2 \leq C_\Phi\qquad\text{and}\qquad
	\|\gt^{(j)}\|_{L^\infty(\Q^*)}^2 \leq C_{\gt}.
\end{displaymath}
Then there exists a constant $C$ independent of $\tau$ such that	
\begin{displaymath}
	\|\bar{p}^{(\ell)}\|_{L^\infty(\cHQ)}^2 \leq \left(C_\phi + \tfrac{C_{\gt}T}{\Kalpha}\right)\mathrm{e}^{CT}
\end{displaymath}
for all $\ell\in\mathbb{N}$. In particular, we obtain $\bar{p}\in W^{2,\infty}(\cHQ)$.
\end{proposition}
\begin{proof}
Using Bellman's method of steps (cf.\ \cite[Ch.\,3.4]{BelZ03}) for \eqref{eqn:neutralDelayPDE}, we consider the sequence of PDEs
\begin{equation}
\label{eqn:neutralDelay:MOS}			
	\langle \mosdp{n},q\rangle + \langle \calK\mosp{n},q\rangle 
	= \langle \omega\, \mosdp{n-1} + \mos{\gt}{n},q\rangle
\end{equation}
for all test functions~$q\in \Q$. Here, $\mosp{n}(t) \vcentcolon=\bar{p}(t+t_{n-1})$ for $t\in[0,\tau]$ denotes the restriction of the solution to the time interval $I_n \vcentcolon= [t_{n-1}, t_{n}]$ with the the convention $\mosp{0}(t) = \Phi(t-\tau)$ for $t\in[0,\tau]$. With the test function $q = \mosp{n}$, the G\aa{}rding inequality~\eqref{eqn:Gaarding}, and the weighted Young's inequality we obtain
\begin{align*}
	\tfrac{1}{2}\tddt \|\mosp{n}\|_{\cHQ}^2 + \Kalpha\|\mosp{n}\|_{\Q}^2 
	&\leq \Kbeta\|\mosp{n}\|_{\cHQ}^2 + \omega\, \|\mosdp{n-1}\|_{\cHQ}\|\mosp{n}\|_{\cHQ} + \|\mos{\gt}{n}\|_{\Q^*}\|\mosp{n}\|_{\Q} \\
	&\leq \Kbeta\|\mosp{n}\|_{\cHQ}^2 + \tfrac{C}{2}\|\mosdp{n-1}\|_{\cHQ}^2 + \tfrac{1}{2\Kalpha}\|\mos{\gt}{n}\|_{\Q^*}^2 + \Kalpha\|\mosp{n}\|_{\Q}^2
\end{align*}
with $C \vcentcolon= \tfrac{\omega^2}{\Kalpha} C_{\Q\hook\cHQ}^{2}$\!. 
Absorbing $\|\mosp{n}\|_{\Q}^2$, integrating over~$[0,t]$, and using Grönwall's inequality yields
\begin{align*}
	\|\mosp{n}(t)\|_{\cHQ}^2 
	\leq \mathrm{e}^{2\tau\Kbeta}\left(\|\mosp{n}(0)\|_{\cHQ}^2 + \int_0^t \left(C\|\mosdp{n-1}(s)\|_{\cHQ}^2 + \tfrac{1}{\Kalpha}\|\mos{\gt}{n}(s)\|_{\Q^*}^2\right)\right).
\end{align*}
The smoothness of $\bar{p}$ implies $\|\mosp{n}(0)\|_{\cHQ}^2 = \|\mosp{n-1}(0)\|_{\cHQ}^2 + \int_0^\tau \tddt \|\mosp{n-1}(s)\|_{\cHQ}^2 \ds$. Using $\tddt \|\mosp{n-1}\|_{\cHQ}^2 \leq \|\mosdp{n-1}\|_{\cHQ}^2 + \|\mosp{n-1}\|_{\cHQ}^2$ and taking the $L^\infty$-norm implies
\begin{multline*}
	\|\mosp{n}\|_{\LinfHQ}^2 \leq \mathrm{e}^{2\tau\Kbeta}(1+\tau)\|\mosp{n-1}\|_{\LinfHQ}^2\\ + \tau\widetilde{C}\|\mosdp{n-1}\|_{\LinfHQ}^2 + \tfrac{\tau\mathrm{e}^{2\tau\Kbeta}}{\Kalpha}\|\mos{\gt}{n}(s)\|_{\LinfQd}^2
\end{multline*}
with $\widetilde{C} = \mathrm{e}^{2\tau\Kbeta}(1+C)$. 
Using the smoothness of $\bar{p}$, we inductively obtain
\begin{multline*}
	\|\mosp{n}\|_{\LinfHQ}^2 
	\leq \mathrm{e}^{2n\tau\Kbeta}\sum_{i=0}^n \binom{n}{i}(1+\tau)^{n-i}(\tau\widetilde{C})^i \|\mos{\bar{p}}{0}^{(i)} \|_{\LinfHQ}^2\\
	+ \tfrac{\tau\mathrm{e}^{2n\tau\Kbeta}}{\Kalpha}\,\sum_{i=0}^{n-1}\sum_{j=0}^i \binom{i}{j} (1+\tau)^{i-j}(\tau \widetilde{C})^j \|\mos{\gt^{(j)}}{n-i}\|_{\LinfQd}^2.
\end{multline*}
From $\|\Phi^{(j)}\|_{L^\infty(-\tau,0;\cHQ)}^2 \leq C_\Phi$ and $\|\gt^{(j)}\|_{L^\infty(\Q^*)}^2 \leq C_{\gt}$ for all $j\in\mathbb{N}$ and $n\tau \leq T$, we deduce
\begin{align*}
	\mathrm{e}^{2n\tau\Kbeta}\sum_{i=0}^n \binom{n}{i}(1+\tau)^{n-i}(\tau\widetilde{C})^i \|\mos{\bar{p}}{0}^{(i)}\|_{\LinfHQ}^2 
	&\leq C_\Phi\, \mathrm{e}^{(1+2\Kbeta + \widetilde{C})T}.
\end{align*}
Similarly, we obtain
\begin{displaymath}
	\tfrac{\tau\mathrm{e}^{2n\tau\Kbeta}}{\Kalpha}\,\sum_{i=0}^{n-1}\sum_{j=0}^i \binom{i}{j} (1+\tau)^{i-j}(\tau \widetilde{C})^j \|\mos{\gt^{(j)}}{n-i}\|_{\LinfQd}^2 
	\leq \tfrac{C_{\gt}\,T}{\Kalpha}\, \mathrm{e}^{(1+2\Kbeta+\widetilde{C})T},
\end{displaymath}
which implies
\begin{displaymath}
	\|\bar{p}\|_{L^\infty(\cHQ)}^2 \leq \left(C_\phi + \tfrac{C_{\gt}T}{\Kalpha}\right)\mathrm{e}^{(1+2\Kbeta + \widetilde{C})T}.
\end{displaymath}
Repeating this procedure with derivatives of equation~\eqref{eqn:neutralDelay:MOS} finishes the proof.
\end{proof}

\begin{thebibliography}{RNM{\etalchar{+}}03}
	
	\bibitem[ACM{\etalchar{+}}19]{AltCMPP19}
	R.~Altmann, E.~Chung, R.~Maier, D.~Peterseim, and S.-M. Pun.
	\newblock Computational multiscale methods for linear heterogeneous
	poroelasticity.
	\newblock {\em J. Comput. Math.}, accepted for publication:1--18, 2019.
	
	\bibitem[AH18]{AltH18}
	R.~Altmann and J.~Heiland.
	\newblock Regularization and {R}othe discretization of semi-explicit operator
	{DAE}s.
	\newblock {\em Int. J. Numer. Anal. Model.}, 15(3):452--478, 2018.
	
	\bibitem[Alt15]{Alt15}
	R.~Altmann.
	\newblock {\em {R}egularization and {S}imulation of {C}onstrained {P}artial
		{D}ifferential {E}quations}.
	\newblock Dissertation, Technische Universit{\"a}t Berlin, 2015.
	
	\bibitem[AZ18]{AltZ18}
	R.~Altmann and C.~Zimmer.
	\newblock On the smoothing property of linear delay partial differential
	equations.
	\newblock {\em J. Math. Anal. Appl.}, 467(2):916--934, 2018.
	
	\bibitem[BBF13]{BofBF13}
	D.~Boffi, F.~Brezzi, and M.~Fortin.
	\newblock {\em Mixed finite element methods and applications}.
	\newblock Springer, Heidelberg, 2013.
	
	\bibitem[BCP96]{BreCP96}
	K.E. Brenan, S.L. Campbell, and L.~R. Petzold.
	\newblock {\em Numerical solution of initial-value problems in
		differential-algebraic equations}.
	\newblock Society for Industrial and Applied Mathematics (SIAM), Philadelphia,
	PA, 1996.
	
	\bibitem[BGZ99]{BelGZ99}
	A.~Bellen, N.~Guglielmi, and M.~Zennaro.
	\newblock On the contractivity and asymptotic stability of systems of delay
	differential equations of neutral type.
	\newblock {\em BIT Numer. Math.}, 39(1):1--24, 1999.
	
	\bibitem[Bio41]{Biot41}
	M.~A. Biot.
	\newblock General theory of three-dimensional consolidation.
	\newblock {\em J. Appl. Phys.}, 12(2):155--164, 1941.
	
	\bibitem[Bio56]{Bio56}
	M.~A. Biot.
	\newblock Thermoelasticity and irreversible thermodynamics.
	\newblock {\em J. Appl. Phys.}, 27:240--253, 1956.
	
	\bibitem[Bra07]{Bra07}
	D.~Braess.
	\newblock {\em Finite Elements - Theory, Fast Solvers, and Applications in
		Solid Mechanics}.
	\newblock Cambridge University Press, New York, third edition, 2007.
	
	\bibitem[BY11]{BY11}
	Z.~Z. Bai and X.~Yang.
	\newblock {On convergence conditions of waveform relaxation methods for linear
		differential-algebraic equations}.
	\newblock {\em J. Comput. Appl. Math.}, 235(8):2790--2804, 2011.
	
	\bibitem[BZ03]{BelZ03}
	A.~Bellen and M.~Zennaro.
	\newblock {\em Numerical Methods for Delay Differential Equations}.
	\newblock Oxford University Press, New York, 2003.
	
	\bibitem[Cam80]{Cam80}
	S.~L. Campbell.
	\newblock Singular linear systems of differential equations with delays.
	\newblock {\em Appl. Anal.}, 11(2):129--136, 1980.
	
	\bibitem[Cia88]{Cia88}
	P.~G Ciarlet.
	\newblock {\em Mathematical elasticity. Vol. {I}}.
	\newblock North-Holland, Amsterdam, 1988.
	
	\bibitem[CR14]{CalR14}
	W.~D. Callister and D.~G. Rethwisch.
	\newblock {\em Materials science and engineering: {A}n introduction}.
	\newblock Wiley, Hoboken, NJ, ninth edition, 2014.
	
	\bibitem[DC93]{DetC93}
	E.~Detournay and A.~H.~D. Cheng.
	\newblock Fundamentals of poroelasticity.
	\newblock In {\em Analysis and design methods}, pages 113--171. Elsevier, 1993.
	
	\bibitem[Dur08]{Dur08}
	R.~G. Dur{\'a}n.
	\newblock Mixed finite element methods.
	\newblock In {\em Mixed Finite Elements, Compatibility Conditions, and
		Applications}, pages 1--44. Springer, 2008.
	
	\bibitem[EM09]{ErnM09}
	A.~Ern and S.~Meunier.
	\newblock A posteriori error analysis of {E}uler-{G}alerkin approximations to
	coupled elliptic-parabolic problems.
	\newblock {\em ESAIM: Math. Model. Numer. Anal.}, 43(2):353--375, 2009.
	
	\bibitem[EM13]{EmmM13}
	E.~Emmrich and V.~Mehrmann.
	\newblock Operator differential-algebraic equations arising in fluid dynamics.
	\newblock {\em Comput. Methods Appl. Math.}, 13(4):443--470, 2013.
	
	\bibitem[Eva98]{Eva98}
	L.~C. Evans.
	\newblock {\em Partial Differential Equations}.
	\newblock American Mathematical Society (AMS), Providence, second edition,
	1998.
	
	\bibitem[FAC{\etalchar{+}}19]{FuACMPP19}
	S.~Fu, R.~Altmann, E.~Chung, R.~Maier, D.~Peterseim, and S.-M. Pun.
	\newblock Computational multiscale methods for linear poroelasticity with high
	contrast.
	\newblock {\em J. Comput. Phys.}, 395:286--297, 2019.
	
	\bibitem[Fu19]{Fu19}
	G.~Fu.
	\newblock A high-order {HDG} method for the {B}iot’s consolidation model.
	\newblock {\em Comput. Math. Appl.}, 77(1):237--252, 2019.
	
	\bibitem[GKC03]{GuKC03}
	K.~Gu, V.~L. Kharitonov, and J.~Chen.
	\newblock {\em Stability of Time-Delay Systems}.
	\newblock Birkh{\"{a}}user, Boston, MA, 2003.
	
	\bibitem[HK18]{HonK18}
	Q.~Hong and J.~Kraus.
	\newblock Parameter-robust stability of classical three-field formulation of
	{B}iot's consolidation model.
	\newblock {\em Electron. Trans. Numer. Anal.}, 48:202--226, 2018.
	
	\bibitem[HKLP19]{HonKLP19}
	Q.~Hong, J.~Kraus, M.~Lymbery, and F.~Philo.
	\newblock Conservative discretizations and parameter-robust preconditioners for
	{B}iot and multiple-network flux-based poroelasticity models.
	\newblock {\em Numer. Linear Algebr.}, 26(4):e2242, 2019.
	
	\bibitem[HRGZ17]{HuRGZ17}
	X.~Hu, C.~Rodrigo, F.~J. Gaspar, and L.~T. Zikatanov.
	\newblock A nonconforming finite element method for the {B}iot’s
	consolidation model in poroelasticity.
	\newblock {\em J. Comput. Appl. Math.}, 310:143--154, 2017.
	
	\bibitem[JCLT19]{JCLT19}
	G.~Jv, M.~Cai, J.~Li, and J.~Tian.
	\newblock Parameter-robust multiphysics algorithms for {B}iot model with
	application in brain edema simulation.
	\newblock ArXiv Preprint 1906.08802, 2019.
	
	\bibitem[KM06]{KunM06}
	P.~Kunkel and V.~Mehrmann.
	\newblock {\em Differential-Algebraic Equations. Analysis and Numerical
		Solution}.
	\newblock European Mathematical Society, Zürich, 2006.
	
	\bibitem[LMW17]{LMW17}
	J.~J. Lee, K.-A. Mardal, and R.~Winther.
	\newblock Parameter-robust discretization and preconditioning of {Biot's}
	consolidation model.
	\newblock {\em SIAM J. Sci. Comput.}, 39(1):A1--A24, 2017.
	
	\bibitem[Mie89]{Mie89}
	U.~Miekkala.
	\newblock Dynamic iteration methods applied to linear {DAE} systems.
	\newblock {\em J. Comput. Appl. Math.}, 25:133--151, 1989.
	
	\bibitem[MP17]{MalP17}
	A.~M{\aa}lqvist and A.~Persson.
	\newblock A generalized finite element method for linear thermoelasticity.
	\newblock {\em ESAIM: Math. Model. Numer. Anal.}, 51(4):1145--1171, 2017.
	
	\bibitem[PW07a]{PhiW07a}
	P.~J. Phillips and M.~F. Wheeler.
	\newblock A coupling of mixed and continuous {G}alerkin finite element methods
	for poroelasticity {I}: the continuous in time case.
	\newblock {\em Comput. Geosci.}, 11(2):131--144, 2007.
	
	\bibitem[PW07b]{PhiW07b}
	P.~J. Phillips and M.~F. Wheeler.
	\newblock A coupling of mixed and continuous {G}alerkin finite element methods
	for poroelasticity {II}: the discrete-in-time case.
	\newblock {\em Comput. Geosci.}, 11(2):145--158, 2007.
	
	\bibitem[PW08]{PhiW08}
	P.~J. Phillips and M.~F. Wheeler.
	\newblock A coupling of mixed and discontinuous {G}alerkin finite-element
	methods for poroelasticity.
	\newblock {\em Comput. Geosci.}, 12(4):417--435, 2008.
	
	\bibitem[RNM{\etalchar{+}}03]{RNM+03}
	T.~Roose, P.~A. Netti, L.~L. Munn, Y.~Boucher, and R.~K. Jain.
	\newblock Solid stress generated by spheroid growth estimated using a linear
	poroelasticity model.
	\newblock {\em Microvasc. Res.}, 66(3):204--212, 2003.
	
	\bibitem[Sho00]{Sho00}
	R.~E. Showalter.
	\newblock Diffusion in poro-elastic media.
	\newblock {\em J. Math. Anal. Appl.}, 251(1):310--340, 2000.
	
	\bibitem[Tar06]{Tar06}
	L.~Tartar.
	\newblock {\em An Introduction to Navier-Stokes Equation and Oceanography}.
	\newblock Springer, Berlin, Heidelberg, 2006.
	
	\bibitem[TU19]{TreU19}
	S.~Trenn and B.~Unger.
	\newblock Delay regularity of differential-algebraic equations.
	\newblock Preprint, submitted for publication, 2019.
	
	\bibitem[TV11]{TulV11}
	B.~Tully and Y.~Ventikos.
	\newblock Cerebral water transport using multiple-network poroelastic theory:
	application to normal pressure hydrocephalus.
	\newblock {\em J. Fluid Mech.}, 667:188--215, 2011.
	
	\bibitem[{Ung}18]{Ung18}
	B.~{Unger}.
	\newblock Discontinuity propagation in delay differential-algebraic equations.
	\newblock {\em Electron. J. Linear Algebr.}, 34:582--601, 2018.
	
	\bibitem[VCT{\etalchar{+}}16]{VarCTHLTV16}
	J.~C. Vardakis, D.~Chou, B.~J. Tully, C.~C. Hung, T.~H. Lee, P.~H. Tsui, and
	Y.~Ventikos.
	\newblock Investigating cerebral oedema using poroelasticity.
	\newblock {\em Med. Eng. Phys.}, 38(1):48--57, 2016.
	
	\bibitem[WG07]{WheG07}
	M.~F. Wheeler and X.~Gai.
	\newblock Iteratively coupled mixed and {G}alerkin finite element methods for
	poro-elasticity.
	\newblock {\em Numer. Meth. Part. D. E.}, 23(4):785--797, 2007.
	
	\bibitem[Zei90]{Zei90a}
	E.~Zeidler.
	\newblock {\em Nonlinear Functional Analysis and its Applications {IIa}: Linear
		Monotone Operators}.
	\newblock Springer-Verlag, New York, 1990.
	
	\bibitem[Zob10]{Zob10}
	M.~D. Zoback.
	\newblock {\em Reservoir Geomechanics}.
	\newblock Cambridge University Press, Cambridge, 2010.
	
\end{thebibliography}
\end{document}